\documentclass[12pt]{article}

\usepackage[colorlinks,
            linkcolor=blue,
            anchorcolor=blue,
            citecolor=blue]{hyperref}

\newif\iffund
\fundtrue

\usepackage{amsmath, amsthm, amssymb, color}
\usepackage{graphicx}
\usepackage{makecell}
\usepackage{multicol}
\usepackage{url}

\usepackage[table,xcdraw]{xcolor}
\usepackage{tikz}
\usetikzlibrary{arrows.meta}
\usepackage{array,amscd}
\usepackage{amsfonts}
\usepackage{appendix}
\usepackage[arrow, matrix, curve]{xy}
\usepackage[comma,sort,authoryear]{natbib}
\usepackage{booktabs}
\usepackage{siunitx}
\usepackage{multirow}
\usepackage[font=small,labelfont=bf]{caption}
\usepackage{caption}
\usepackage{subcaption}
\usepackage{blkarray}
\usepackage{algorithm}
\usepackage{algpseudocode}

\usepackage{tikz}
\usetikzlibrary{decorations}
\usetikzlibrary {shapes}
\usetikzlibrary {arrows}
\usetikzlibrary {positioning}
\usetikzlibrary {calc}

\usepackage{enumitem} %for roman index enumeration?
\setenumerate{leftmargin=*}
\usepackage[margin=3cm]{geometry}
\allowdisplaybreaks[3]	

\newtheorem{theorem}{Theorem}[section]
\newtheorem{corollary}{Corollary}[section]
\newtheorem{proposition}{Proposition}[section]

\theoremstyle{remark}
\newtheorem{remark}{Remark}[section]
\newtheorem{example}{Example}[section]
\theoremstyle{definition}
\newtheorem{definition}{Definition}[section]

\newcommand{\RR}{\mathbb{R}}

\definecolor{darkgreen}{rgb}{0,0.4,0}
\definecolor{MyBlue}{rgb}{0,0.08,0.7} 
\definecolor{MyRed}{rgb}{0.85,0.08,0}

\newcommand{\toblue}{{\, \color{MyBlue}\to}\,}

\renewcommand{\toblue}{\to}

\newcommand\trans{T}
\DeclareMathOperator{\var}{Var}
\DeclareMathOperator{\tr}{tr}

\DeclareMathOperator{\diag}{diag}
\DeclareMathOperator{\pa}{pa}
\DeclareMathOperator{\ch}{ch}
\DeclareMathOperator{\de}{de}
\DeclareMathOperator{\an}{an}

\DeclareMathOperator{\ci}{\perp\kern-1.3ex\perp}
\DeclareMathOperator{\nci}{\not\kern-0.3ex\ci}

\DeclareMathOperator*{\argmin}{argmin}

\makeatletter
\renewcommand*\env@matrix[1][\arraystretch]{%
  \edef\arraystretch{#1}%
  \hskip -\arraycolsep
  \let\@ifnextchar\new@ifnextchar
  \array{*\c@MaxMatrixCols c}}
\makeatother

\providecommand{\keywords}[1]
{
  \small	
  \textbf{\textit{Keywords:}} #1
}

\usepackage{authblk}
\title{Partial Homoscedasticity in Causal Discovery with Linear Models}
\author[]{Jun Wu, Mathias Drton}
\affil[]{\normalsize Technical University of Munich, Germany; \\
TUM School of Computation, Information and Technology, \\
Department of Mathematics and Munich Data Science Institute}
\date{}

\begin{document}

\maketitle

\begin{abstract}
Recursive linear structural equation models and the associated directed acyclic graphs (DAGs) play an important role in causal discovery.  The classic identifiability result for this class of models states that when only observational data is available, each DAG can be identified only up to a Markov equivalence class.  In contrast, recent work has shown that the DAG can be uniquely identified if the errors in the model are homoscedastic, i.e., all have the same variance.  This equal variance assumption yields methods that, if appropriate, are highly scalable and also sheds light on fundamental information-theoretic limits and optimality in causal discovery.  In this paper, we fill the gap that exists between the two previously considered cases, which assume the error variances to be either arbitrary or all equal. Specifically, we formulate a framework of partial homoscedasticity, in which the variables are partitioned into blocks and each block shares the same error variance.  
For any such groupwise equal variances assumption, we characterize when two DAGs give rise to identical Gaussian linear structural equation models.  Furthermore, we show how the resulting distributional equivalence classes may be represented using a completed partially directed acyclic graph (CPDAG), and we give an algorithm to efficiently construct this CPDAG.  In a simulation study, we demonstrate that greedy search provides an effective way to learn the CPDAG and exploit partial knowledge about homoscedasticity of errors in structural equation models.
\end{abstract}


\section{Introduction}
\label{sec:intro}

A structural equation model (SEM) describes the stochastic dependence among a group of random variables in terms of noisy functional relationships between causes and their effects.  
In their interpretation as causal models, SEMs furnish models of the variables' joint distribution not only in observational studies but also under experimental interventions, providing a rigorous definition of the causal effect of an intervention \citep{pearl:2009,spirtes:2000}.  In this realm, the problem of causal discovery, also known as structure learning, is the problem of learning a causal model from available data by determining the direct causes of each variable with the help of statistical methods \citep{drton:maathuis:2017}.  For this purpose, it is convenient to take a point of view of probabilistic graphical modeling and represent each SEM with the help of a directed graph, with vertices corresponding to the random variables at hand, and directed edges linking the functionally related variables \citep{handbook}.  We tacitly assume all considered SEMs to be recursive, i.e., the underlying graph is a directed acyclic graph (DAG).

Each DAG uniquely encodes a causal model, but when only observational data are available, different DAGs (i.e., different systems of structural equations) may be equivalent in the sense of defining the same statistical model for the observations.   The target of inference then
becomes an equivalence class of empirically indistinguishable SEMs.  Hallmark theory for graphical models provide a characterization of SEMs in terms of conditional independence, show how this characterization leads to efficient decision criteria for the observational equivalence of two SEMs given by two different DAGs, and finally develop a graphical representation of the resulting Markov equivalence class by means of a completed partially directed acyclic graph (CPDAG).  This theory is summarized, for instance, in \citet{handbook:studeny}.  

The theory just mentioned holds for fully nonparametric models.  Somewhat amazingly, the same arguments and results also apply to the widely considered special case of linear SEMs with Gaussian errors with arbitrary variances, where again only a Markov equivalence class may be inferred from observational data; see, e.g., the review of \citet{drton:algebraic}.  However, restricted non-Gaussian models behave differently and typically lead to a setting where every DAG defines a unique model for observational data.   This has been shown, for example, for linear models with non-Gaussian errors \citep{Shimizu06alinear,lingam:book} and nonlinear additive noise models \citep{NIPS2008_f7664060,Peters11identifiabilityof}.  Similarly, additional assumptions may render linear Gaussian SEMs identifiable.  The most prominent example in this direction are linear SEMs with Gaussian errors that all share a common variance \citep{Peters2014biometrika}.  Despite its restrictive nature, the equal variance setting plays a useful role as a test bed for developing scalable causal discovery methods \citep{ghoshal2017learning,chen2018causal,pmlr-v84-ghoshal18a,gao2020polynomial} and deriving fundamental information-theoretic limits and optimality results \citep{pmlr-v151-gao22a}; for further extensions see \citet{park2020gauss,park2020add,AMP}.
%
%.  Analogous idea extends to identifiability results for linear models with heterogeneous errors \citep{park2020gauss}, additive %noise models \citep{park2020add}, and even Andersson-Madigan-Perlman (AMP) chain graph models \citep{AMP}.

In this paper, we consider the %unexplored 
realm of linear Gaussian SEMs that fall between the classical case with arbitrary error variances and the case with all error variances equal.  
% While we focus on the well-known setting of recursive linear models and Gaussian errors, we introduce a novel additional assumption: 
To this end we formulate a novel framework:  partial homoscedasticity of the errors.  Specifically, we encode the modeling assumptions in a partition over the variables, with the interpretation that the errors associated to the variables in the same partition block have equal variances.  In this framework, the minimal and maximal partitions recover the previously studied cases.  The minimal partition in which every variable forms a block of size one recovers the classical case of arbitrary error variances.  The maximal partition in which all variables are in one block gives the equal variance case.  

Our results first provide an implicit description of partially homoscedastic linear Gaussian SEMs, with a partition specifying groupwise equal variances.  This description is based on conditional independence constraints as well as constraints we deduce from the equalities of error variances.  We then characterize when two DAGs define the same partially homoscedastic linear Gaussian SEM.  This characterization shows that because of the equal variance constraints, the existence of a pair of variables in the same block of the considered partition determines the orientation of the edges that have these two nodes as endpoints.  Moreover, we generalize the concept of a CPDAG to the partially homoscedastic case and provide an algorithm for efficient construction of the CPDAG.  This algorithm is an adjusted version of the general algorithm for constructing an equivalence class under background knowledge \citep{meek1995causal}.

The remainder of this paper is organized as follows: In Section \ref{sec_bg}, we introduce basic notation and necessary background for linear SEMs and their representation using DAGs. In Section \ref{sec_ghe}, we discuss the partially homoscedastic setup with groupwise equal error variances and derive the equal variance constraints that are needed in an implicit description of the models.   In Section \ref{sec_eq}, we characterize the equivalence classes in our setup and give an algorithm to construct the CPDAG. In Section \ref{sec_gs} we propose a greedy search scheme for model selection based on information criteria, which is seen to be effective in a simulation study.  We conclude with a discussion in Section \ref{sec_discuss}.

\section{Background}\label{sec_bg}
\subsection{Linear structural equation models}

A linear structural equation model (SEM) is specified via a linear system consisting of equations among variables $\{X_i: i\in V\}$ and random errors $\{\varepsilon_i: i\in V\}$, where $V$ is an index set of size $|V|=p$.  In terms of the random vectors ${X}:=(X_i)_{i\in V}$ and ${\epsilon}:=(\varepsilon_i)_{i\in V}$, the linear system can be written as
\begin{align}\label{sem}
{X}=\Lambda^T {X}+{\epsilon},
\end{align}
where $\Lambda=(\lambda_{ij})\in\RR^{V\times V}$ is a matrix of coefficients representing the causal structure among the variables.
Assuming that the matrix $I-\Lambda$ is invertible ($I$ is the identity matrix), the linear equation system \eqref{sem} is solved uniquely by $X=(I-\Lambda)^{-T}\epsilon$, with covariance matrix 
\begin{align}\label{cov}
  \var[X] = (I - \Lambda)^{-\trans} \Omega (I -
  \Lambda)^{-1}, 
\end{align}
where $\Omega$ is the covariance matrix of the random vector of stochastic errors $\epsilon$. We assume that the errors are independent such that $\Omega=\diag(\boldsymbol{\omega})$ is a diagonal matrix with all diagonal entries positive.  Any specific SEM makes the assumption that a subset of the entries of $\Lambda$ is zero.  Such an assumption is naturally encoded in a directed graph.  We will detail the connection between a linear SEM and its directed graph in Section \ref{sec:2.3}.

\subsection{Graph terminology}
\label{subsec:graph_terminology}

A directed graph $G=(V,E)$ is a tuple that pairs a set of nodes $V$ with a set of edges $E\in V\times V$. The elements in $E$ are ordered pairs of the form $(i,j)$, $i\neq j$, encoding the edge $i\rightarrow j$ in the graph. We say that the edge $(i,j)$ is an outgoing edge from the arrow tail $i$ and an incoming edge to the arrow head $j$. The node $i$ is a \textit{parent} of $j$ and the node $j$ is a \textit{child} of $i$. We denote the sets of all parents and children of a node $i$ by $\pa(i)$ and $\ch(i)$, respectively. Similarly, the notation $\an(i)$ denotes the set of \textit{ancestors} of $i$ and $\de(i)$ denotes the set of \textit{descendants} of $i$. For simplicity, we adopt the convention that $i\notin \an(i)$ and $i\in\de(i)$. If $\an(i)\cap\de(i)=\emptyset$, then node $i$ does not lie on any directed cycle.  If $\an(i)\cap \de(i)=\emptyset$ 
for every node $i$, then the graph is a \textit{directed acyclic graph} (DAG). To distinguish those sets in different graphs, we use graph index as the subscript: $\pa_{G}$, etc.

A \textit{collider triple} in a directed graph $G=(V,E)$ is a triple of vertices $(i,j,k)$ such that there are edges between $i$ and $j$ and between $j$ and $k$, with $j$ being a head on both these edges ($i\rightarrow j \leftarrow k$). If there exists an edge between $i$ and $j$, i.e., $(i,j)\in E$ or $(j,i)\in E$, then the middle node $j$ is a \textit{shielded collider} in this collider triple. Otherwise, the node $j$ is an \textit{unshielded collider}. 

In a directed graph $G=(V,E)$, a \textit{path} is an alternating sequence of nodes from $V$ and edges from $E$, such that each edge in the sequence is an edge between the nodes that precede and succeed it.
Note that this definition allows a path to contain a node more than once.
Given a fixed set $S\subseteq V$, two nodes $i,j\notin S$ are \textit{d-connected} by $S$ if $G$ contains a path from $i$ to $j$ that has all colliders in $S$ and all non-colliders outside $S$. If there exists no such path, then $i$ and $j$ are \textit{d-separated} by $S$, denoted as 
$i\perp_d j \mid S$.
A \textit{trek} is a path without collider triples, making its endpoints d-connected given $\emptyset$.  A trek takes the form
\begin{align*}
i_l^{L} \leftarrow \dots \leftarrow i_1^L \leftarrow i_0  \rightarrow i_1^R \rightarrow\dots \rightarrow i_{r}^{R},
\end{align*}
where $i_0$ is called the top node.  This top node is characterized by not being the head of any edge on the trek.
Every trek has a left hand side and a right hand side, corresponding to the nodes with superscript $L$ or $R$. By convention, the top node $i_0$ is included in both the left hand and the right hand side of the trek. 

\subsection{Graphical models, trek rule and conditional independence}\label{sec:2.3}

As noted earlier, any specific linear SEM is obtained by requiring a subset of the entries in the parameter matrix $\Lambda=(\lambda_{ij})$ in \eqref{sem} to vanish.  This requirement may be represented by drawing an edge $i\to j$ whenever the matrix entry $\lambda_{ij}$ is \emph{not} constrained to be zero.

Let $G=(V,E)$ be a DAG. Let $\mathbb{R}^E$ be the set of real $V\times V$-matrices with support in $E$, that is, 
\begin{equation}
  \label{eq:R^E}
      \mathbb{R}^E = \big\{\,\Lambda=(\lambda_{ij}) \in \mathbb{R}^{V
      \times V} : \lambda_{ij} = 0 \ \text{ if } \ i \toblue j \notin E
    \,\big\}.
\end{equation}
Then the DAG $G$ encodes the linear SEM with independent errors whose coefficient matrix is constrained to have a zero pattern given by $\mathbb{R}^E$.  In the resulting linear SEM the covariance matrix is parametrized through the map
\begin{align*}
    \phi_G :  \mathbb{R}^E\times (0,\infty)^V 
    &\mapsto \mathit{PD}, \\
    (\Lambda, \boldsymbol{\omega}) &\mapsto (I - \Lambda)^{-T} \diag(\boldsymbol{\omega}) (I - \Lambda)^{-1},\nonumber
\end{align*}
where $\mathit{PD}$ denotes the {cone of positive definite matrices}.  Note that for a DAG, the matrix $I-\Lambda$ is invertible for all $\Lambda\in\mathbb{R}^E$ because the row and columns of $\Lambda$ can be permuted to bring the matrix in lower-triangular form.

\begin{definition}
The \emph{linear Gaussian model} given by a DAG $G=(V,E)$ is the family of all multivariate normal distributions on $\mathbb{R}^V$ with covariance matrix in the set
$$
M_{G}=\left\{\Sigma: \Sigma=\phi_G(\Lambda,\boldsymbol{\omega}),\ \Lambda\in \mathbb{R}^E ,\  \boldsymbol{\omega}\in (0,\infty)^V \right\}.
$$
\end{definition}

The map $\phi_G$ computes the covariance matrix of the random vector $X$ from the coefficient matrix $\Lambda$ and the vector of error variances $\boldsymbol{\omega}$.  A classical result known as the \emph{trek rule} provides a combinatorial description of the coordinates of $\phi_G$, i.e., of the individual covariances between entries of $X$; see, e.g., Theorem 4.1 in the review of \citet{drton:algebraic}. 

\begin{theorem}[Trek rule]\label{trekrule}
Let $G=(V,E)$ be a DAG, and let $\Lambda\in \mathbb{R}^{E}$ and $\boldsymbol{\omega}\in (0,\infty)^V$. For $i,j\in V$, let $\mathcal{T}(i, j)$ be the set of all treks between $i$ and $j$.  For a trek $\tau$ with top node $i_0$, we define the trek monomial
\begin{align*}
\tau(\Lambda, \boldsymbol{\omega})=\boldsymbol{\omega}_{i_{0}} \prod_{k \rightarrow l \in \tau} \lambda_{k l}. 
\end{align*}
Then the covariance between $X_i$ and $X_j$ equals the sum of all trek monomials for treks between $i$ and $j$, i.e.,
\begin{align*}
\phi_{G}(\Lambda, \boldsymbol{\omega})_{i j}=\sum_{\tau \in \mathcal{T}(i, j)} \tau(\Lambda, \boldsymbol{\omega}), \quad i, j \in V.
\end{align*}
\end{theorem}

SEMs naturally lead to conditional independence constraints and these may be read off from an underlying DAG by inspecting $d$-separation relations, which we recalled in Section~\ref{subsec:graph_terminology}.
Furthermore, in linear Gaussian SEMs, conditional independence corresponds to an algebraic constraint on the covariance matrix of the distribution.  We recall these facts in the following theorem; see \citet[Section 10]{drton:algebraic} or also  \citet[Section 8]{richardson:spirtes:2002} for further discussion and pointers to the original literature.

\begin{theorem}
\label{thm:conditional_independence_dsep}
Let $G=(V,E)$ be a DAG.  Let $i,j$ be two distinct nodes, and let $S\subseteq V\setminus\{i,j\}$.  
\begin{itemize}
    \item[(i)] If $X$ is a multivariate normal random vector with covariance matrix $\Sigma$, then the conditional independence $X_i\ci X_j\mid X_S$ holds if and only if $\det(\Sigma_{iS,jS})=0$.
         \item[(ii)]    
    The conditional independence constraint $\det(\Sigma_{iS,jS})=0$ holds for all covariance matrices $\Sigma\in M_G$ if and only if the $d$-separation $i\perp_d j\mid S$ holds in $G$.
    \item[(iii)] A matrix $\Sigma\in\mathit{PD}$ is in $M_G$ if and only if $\det(\Sigma_{iS,jS})=0$ for all triples $(i,j,S)$ with $i\perp_d j\mid S$ in $G$.
\end{itemize}
\end{theorem}

\section{Partial homoscedasticity}\label{sec_ghe}

\subsection{Setup}

We now introduce a class of models that incorporate partial knowledge about equality among the variances in $\boldsymbol{\omega}=(\omega_i)_{i\in V}$ of the errors in the linear SEM given by a DAG $G=(V,E)$.  We encode this partial knowledge in a partition of the vertex set $V$.

\begin{definition}
Let $\Pi=\{\pi_1,\dots,\pi_K\}$ be a family of non-empty subsets of $V$.  Then $\Pi$ is a partition of $V$ if $\pi_1,\dots,\pi_K$ are pairwise disjoint and $\cup_{k=1}^K \pi_k=V$.  The sets $\pi_1,\dots,\pi_K$ are the \emph{blocks} of the partition.  Corresponding to $\Pi$ is the equivalence relation that has $i,j\in V$  equivalent if $i,j$ are in the same block of $\Pi$;  we then write $i\sim_\Pi j$.
\end{definition}

In order to model a priori assumptions about the equality of some pairs of variances in $\boldsymbol{\omega}=(\omega_i)_{i\in V}$, we are led to consider the partition $\Pi$ such that $i\sim_\Pi j$ precisely when the equality $\omega_i=\omega_j$ is implied by our a priori knowledge.  Each one of the proposed partially homoscedastic linear SEMs is thus associated to a pair $(G,\Pi)$, where $G=(V,E)$ is a DAG and $\Pi$ is a partition of $V$.  In this model, each block of $\Pi$ groups a set of nodes that index error variances that are constrained to be equal to each other.  In other words, we have homoscedasticity of the errors within each partition block, but possibly different variances between the blocks.

\begin{definition}
Let $G=(V,E)$ be a DAG, and let $\Pi$ be a partition of $V$.  
The \emph{partially homoscedastic linear Gaussian model} given by the pair $(G,\Pi)$ is the family of all multivariate normal distributions on $\mathbb{R}^V$ with covariance matrix in the set
$$
M_{G,\Pi}=\left\{\Sigma: \Sigma=\phi_G(\Lambda,\boldsymbol{\omega}),\ \Lambda\in \mathbb{R}^E ,\  \boldsymbol{\omega}\in (0,\infty)^V \text{ \rm with } \omega_{ii}=\omega_{jj} \text{ \rm if } i\sim_\Pi j\right\}.
$$
Given a partition $\Pi$, we call two DAGs $G_1$ and $G_2$ model equivalent if they induce the same partially homoscedastic linear Gaussian model, i.e., if $M_{G_1,\Pi}=M_{G_2,\Pi}$.
\end{definition}

From the point of view of causal discovery, assumptions of (partial) homoscedasticity are interesting as extra constraints on error variances lead to a refinement of the Markov equivalence classes that result from conditional independence relations only. 
We exemplify this point here.

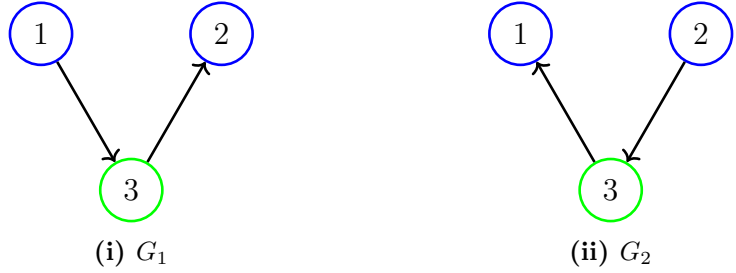
\begin{figure}[t]
\centering
\begin{minipage}[t]{0.4\textwidth}
\centering
\begin{tikzpicture}
[node1/.style={circle,fill=white,draw=blue,minimum size=2 em,inner sep=2pt, line width=1pt},
node2/.append 
style={circle,fill=white,draw=green,minimum size=2 em,inner sep=2pt, line width=1pt}]
\node[node1] (1) at (-2.4, 0) {$1$};
\node[node1] (2) at (2*1.2-2.4, 0) {$2$};
\node[node2] (3) at (1.2-2.4, -1.732*1.2) {$3$};

\draw [->, line width=1pt] (1) to (3);
\draw [->, line width=1pt] (3) to (2);

\end{tikzpicture}
\subcaption{$G_1$}
\end{minipage}
\begin{minipage}[t]{0.4\textwidth}
\centering
\begin{tikzpicture}
[node1/.style={circle,fill=white,draw=blue,minimum size=2 em,inner sep=2pt, line width=1pt},
node2/.append 
style={circle,fill=white,draw=green,minimum size=2 em,inner sep=2pt, line width=1pt}]
\node[node1] (1) at (-2.4, 0) {$1$};
\node[node1] (2) at (2*1.2-2.4, 0) {$2$};
\node[node2] (3) at (1.2-2.4, -1.732*1.2) {$3$};

\draw [->, line width=1pt] (3) to (1);
\draw [->, line width=1pt] (2) to (3);

\end{tikzpicture}
\subcaption{$G_2$}
\end{minipage}
\caption{Under the constraint $\omega_{11}=\omega_{22}$, the two Markov equivalent DAGs $G_1$ and $G_2$ generate different partially homoscedastic linear Gaussian models.}
\label{3nodes_eg}
\end{figure}

\begin{example}
Let $G_1$ and $G_2$ be the two DAGs in Figure~\ref{3nodes_eg}.  Consider first the finest partition $\Pi_{\min}=\{\{1\},\{2\},\{3\}\}$, which encodes that the error variances may be arbitrary positive numbers.  Then $M_{G_1,\Pi_{\min}}=M_{G_1}=M_{G_2}=M_{G_2,\Pi_{\min}}$.  Indeed, the two DAGs $G_1$ and $G_2$ are Markov equivalent, meaning that they encode the same conditional independence relations.  Both graphs feature precisely one d-separation relation, namely, that nodes 1 and 2 are d-separated by 3.  The model may thus be defined by imposing the conditional independence of $X_1$ and $X_2$ given $X_3$, which under Gaussianity is equivalent to the constraint that the covariance matrix $\Sigma=(\sigma_{ij})$ satisfies the polynomial equation  $\sigma_{12}\sigma_{23}-\sigma_{12}\sigma_{33}=0$ (recall Theorem~\ref{thm:conditional_independence_dsep}).

Now let us change the partition to $\Pi=\{\{1,2\},\{3\}\}$, i.e., we assume that $\omega_{11}=\omega_{22}$.  One can then show that the model given by $G_1$ is the set of $3\times 3$ covariance matrices
\[
{M}_{G_1,\Pi}=\{\,\Sigma\in \mathit{PD} \::\: \sigma_{11}\sigma_{33}=\sigma_{22}\sigma_{33}-\sigma_{23}^2,\  \sigma_{13}\sigma_{23}-\sigma_{12}\sigma_{33}=0\, \},
 \]
 whereas $G_2$ defines
 \[
{M}_{G_2,\Pi}=\{\,\Sigma\in \mathit{PD} \::\:
 \sigma_{22}\sigma_{33}=\sigma_{11}\sigma_{33}-\sigma_{13}^2,\ \sigma_{13}\sigma_{23}-\sigma_{12}\sigma_{33}=0\,\}.
 \]
 Both ${M}_{G_1,\Pi}$ and ${M}_{G_2,\Pi}$ have dimension 4, but they are different and their intersection is of lower dimension.
\end{example}

In the remainder of this section, we develop a more general algebraic description of partially homoscedastic linear Gaussian models.  This description furnishes the basis for solving the problem of deciding model equivalence, as developed in Section~\ref{sec_eq}.

\subsection{Equal variance constraints}\label{eqvar_csts}

The key difference between partially homoscedastic models and the classical case of arbitrary Gaussian errors is the emergence of constraints due to the equalities among error variances.  To exbhibit these constraints we first review how an error variance can be identified from the covariance matrix of the observations.

\begin{theorem}\label{wii_thm}
Let $G=(V,E)$ be a DAG, and let $\Sigma=\phi_G(\Lambda,\boldsymbol{\omega})$ for $\Lambda\in\mathbb{R}^E$ and $\boldsymbol{\omega}\in(0,\infty)^V$.  Then for any $i\in V$, the error variance $\omega_{ii}$ can be computed from the covariance matrix $\Sigma$ as
\begin{align}\label{w-o_eq}
\omega_{ii}=\sigma_{ii}-\Sigma_{i,A}(\Sigma_{A,A})^{-1}\Sigma_{A,i},
\end{align}
where $A$ may be taken to be any subset with $\pa(i)\subseteq A\subseteq V\backslash\de(i)$.
\end{theorem}

\begin{proof} 
We adapt the proof of Theorem 7.1 in \citet{drton:algebraic}, where $A=\pa(i)$.
If a trek between $i$ and $j$ ends at $i$ with an edge of the form $k\leftarrow i$, then the trek is a directed path from $i$ to $j$ and $j\in\de(i)$. Now since $A\subseteq V\backslash\de(i)$, every trek between $i$ and a node in $A$ must end with an edge $k\rightarrow i$. By Theorem \ref{trekrule} we have
\begin{align*}
\Sigma_{A,i}&=\Sigma_{A,\pa(i)}\Lambda_{\pa(i),i}=\Sigma_{A,A}\Lambda_{A,i},
\end{align*}
where the second equality comes from the fact that $\pa(i)\subseteq A$ and $\Lambda_{ki}=0$ for $k\notin\pa(i)$.  In addition, these zeroes in $\Lambda_{A,i}$ imply that
\[
\sigma_{ii}=\omega_{ii}+\Lambda_{\pa(i),i}^T\Sigma_{\pa(i),\pa(i)}\Lambda_{\pa(i),i}=\omega_{ii}+\Lambda_{A,i}^T\Sigma_{A,A}\Lambda_{A,i}.
\qedhere
\]
\end{proof}

We immediately obtain the following corollary for
an equal variance assumption.
\begin{corollary}\label{wij_eq_col}
If two random errors $\epsilon_i$ and $\epsilon_j$ have equal variances, i.e., $i$ and $j$ are in the same block of a considered partition $\Pi$, then all covariance matrices in $M_{G,\Pi}$ satisfy that
\begin{align}\label{wij_eq}
\sigma_{ii}-\Sigma_{i,A_i}(\Sigma_{A_i,A_i})^{-1}\Sigma_{A_i,i}=\sigma_{jj}-\Sigma_{j,A_j}(\Sigma_{A_j,A_j})^{-1}\Sigma_{A_j,j}
\end{align}
for all subsets $A_i$ and $A_j$ such that $\pa(i)\subseteq A_i\subseteq V\backslash\de(i)$ and $\pa(j)\subseteq A_j\subseteq V\backslash\de(j)$.
\end{corollary}

Theorem \ref{wii_thm} admits the following converse.

\begin{theorem}\label{wii-rev_thm}
Let $G=(V,E)$ be a DAG, and let $i\in V$ be one of its nodes.  Let $A\subseteq V\setminus\{i\}$.  Fix any vector of positive error variances $\boldsymbol{\omega}\in(0,\infty)^V$.  If for all $\Lambda\in\mathbb{R}^E$ the matrix $\Sigma=\phi_G(\Lambda,\boldsymbol{\omega})$ satisfies
equation \eqref{w-o_eq}, then it must hold that $\pa(i)\subseteq A\subseteq V\backslash\de(i)$.
\end{theorem}
\begin{proof}
Suppose there exists a node $k\in\pa(i)\backslash A$.  Choose $\Lambda$ to have all entries zero except for $\lambda_{ki}$.  For this choice, the trek rule in Theorem \ref{trekrule} implies that $\Sigma_{i,A}=0$ and, thus, the right hand side of \eqref{w-o_eq} is equal to $\sigma_{ii}$.  But the trek rule also yields that $\sigma_{ii}=\omega_{ii}+\lambda_{ki}^2\omega_{kk}>\omega_{ii}$, which contradicts the assumption that \eqref{w-o_eq} holds.  We conclude that $\pa(i)\subseteq A$.

Next, suppose that there exists a node $k\in A\backslash(V\backslash\de(i))=\de(i)\cap A$. Then $G$ contains a (non-trivial) directed path from $i$ to $k$.  Without loss of generality, we may assume that all interior nodes on the path between $i$ and $k$ are not in $A$. Indeed, we can always pick $k$ to be the first node in $A$ that lies on the path.  So the path is of the form $i\rightarrow m_1\rightarrow \cdots \rightarrow m_t\rightarrow k$ with $m_1,\dots,m_t\notin A$.
Now, take $\Lambda$ with all entries zero except $\lambda_{im_1},\lambda_{m_1 m_2},\dots, \lambda_{m_{t-1}m_t},\lambda_{m_t k}$. The trek rule in Theorem \ref{trekrule} asserts that $\sigma_{ii}=\omega_{ii}$ under this parameterization (every trek between $i$ and $i$ has at least one edge with zero edge weight).  But then equation \eqref{w-o_eq} becomes
\begin{align*}
\omega_{ii}&=\sigma_{ii}-\Big(\lambda_{im_1}\lambda_{m_tk}\prod_{s=2}^t\lambda_{m_{s-1}m_s}\Big)^2[(\Sigma_{A,A})^{-1}]_{kk}\\
&=\sigma_{ii}-\Big(\lambda_{im_1}\lambda_{m_tk}\prod_{s=2}^t\lambda_{m_{s-1}m_s}\Big)^2\frac{1}{\sigma_{kk}}<\sigma_{ii}=\omega_{ii},
\end{align*}
which is again a contradiction.  We conclude that $A\subseteq V\backslash\de(i)$.
\end{proof}

Combining Theorems~\ref{wii_thm} and~\ref{wii-rev_thm}, we can characterize equal variance constraints that hold in a partially homoscedastic linear model. Theorem \ref{iden_thm} is the characterization. For the proof, see Appendix \ref{A1}.

\begin{theorem}\label{iden_thm}
Let $G=(V,E)$ be a DAG, and let $\Pi$ be a partition of the vertex set $V$.  Suppose $i\sim_\Pi j$ are two distinct nodes that lie in the same block of $\Pi$, and let $A_i\subseteq V\setminus\{i\}$ and $A_j\subseteq V\setminus\{j\}$. Then 
the equation \eqref{wij_eq} holds for all matrices $\Sigma\in M_{G,\Pi}$ if and only if $\pa(i)\subseteq A_i\subseteq V\backslash\de(i)$ and $\pa(j)\subseteq A_j\subseteq V\backslash\de(j)$.
\end{theorem}

Theorem \ref{iden_thm} yields a full algebraic characterization of partially homoscedastic linear Gaussian models. Every equal variance condition corresponds to a collection of equations between conditional variances, in which conditioning sets may be taken from a range of sets.  The different conditioning sets will ultimately lead to equivalent constraints once the equal variance constraints are combined with conditional independence constraints.

We now record an observation that will be important for later considerations of model equivalence.  It refers to the smallest and largest conditioning sets, where we partially order sets by set inclusion and extend the ordering lexicographically to pairs of sets; i.e., $(A_i,A_j)\le (B_i,B_j)$ if $A_i\subsetneq B_i$ or if $A_i=B_i$ and $A_j\subseteq B_j$.

\begin{corollary}
\label{cor:Aij}
Let $G$ be a DAG, and let $\Pi$ be a partition of $V$ such that $i\sim_\Pi j$ are in the same block of the partition.  Let $\mathcal{A}_{ij}$ be the family of all pairs $(A_i,A_j)$ with $A_i\subseteq V\setminus\{i\}$ and $A_j\subseteq V\setminus\{j\}$ for which equation~\eqref{wij_eq}, i.e., \[
\sigma_{ii}-\Sigma_{i,A_i}(\Sigma_{A_i,A_i})^{-1}\Sigma_{A_i,i}=\sigma_{jj}-\Sigma_{j,A_j}(\Sigma_{A_j,A_j})^{-1}\Sigma_{A_j,j},\]
holds for all covariance matrices $\Sigma\in M_{G,\Pi}$.  Then
\begin{itemize}
    \item[(i)] $\mathcal{A}_{ij}$ contains a unique minimal pair, namely, $A_i=\pa(i)$ and $A_j=\pa(j)$, and
    \item[(ii)] $\mathcal{A}_{ij}$ contains a unique maximal pair, namely, $B_i=V\backslash\de(i)$ and $B_j=V\backslash\de(j)$.
\end{itemize}
\end{corollary}

\subsection{Characterization of the models}

In Section \ref{eqvar_csts} we considered a class of algebraic constraints that require equality of conditional variances, and we characterized which of these constraints hold in a given partially homoscedastic linear model.
As we show in this section, combining the variance constraints with conditional independence constraints from $d$-separation relations yields an implicit algebraic description of partially homoscedastic linear models.  We begin by revisiting the original proof of \citet[Theorems 1 and 3]{geiger1990logic} for soundness and completeness of $d$-separation in SEMs, which via a slight modification also applies to our specialized setting.  In other words, we clarify in the following proposition that equal variance assumptions do not alter the set of conditional independence relations in a linear SEM.

\begin{proposition}\label{ci_sndcomp}
Let $G=(V,E)$ be a DAG, and let $\Pi$ be a partition of $V$.  Let $i,j$ be two distinct nodes, and let $S\subseteq V\setminus\{i,j\}$.  Then the conditional independence $X_i\ci X_j\mid X_S$ holds for all multivariate normal random vectors $X$ with covariance matrix in $M_{G,\Pi}$ if and only if the $d$-separation $i\perp_d j\mid S$ holds in $G$.
\end{proposition}

\begin{proof}
The ``if'' follows from Theorem~\ref{thm:conditional_independence_dsep} because $M_{G,\Pi}\subseteq M_G$.

For the ``only if'', suppose that $i$ and $j$ are not $d$-separated by $S$.  We then have to construct an example of $\Sigma\in M_{G,\Pi}$ in which the conditional independence does not hold, i.e., $\det(\Sigma_{iS,jS})\not=0$.  To this end, we may slightly modify an example of \citet{geiger1990logic}.  The modification uses equal error variances to ensure $\Sigma$ is in $M_{G,\Pi}$ and not merely in $M_G$. We provide the construction details in Appendix \ref{A2}.
\end{proof}

All ingredients in place, we can now fully describe a partially homoscedastic linear Gaussian model in terms of conditional independence relations and equal variance constraints.  Evidently, all constraints are algebraic, i.e., can be expressed in terms of polynomials.

\begin{theorem}\label{ideal}
Let $G=(V,E)$ be a DAG, and let $\Pi$ be a partition of $V$.  Then a covariance matrix $\Sigma\in\mathit{PD}$ is in the partially homoscedastic linear model $M_{G,\Pi}$ if and only if $\Sigma$ satisfies all 
conditional independence constraints given by $d$-separations and all % the 
equal variance constraints from Corollary~\ref{wij_eq_col}.
\end{theorem}
\begin{proof}
The ``if'' follows from Proposition~\ref{ci_sndcomp} and Corollary~\ref{wij_eq_col}.  For the ``only if'',  let $\Sigma$ satisfy all conditional independence and equal variance constraints associated to $G$.  By Theorem~\ref{thm:conditional_independence_dsep}(iii), a covariance matrix that satisfies all conditional independence constraints given by $d$-separation has to be an element of $M_G$.  
Hence, there exist $\Lambda\in\mathbb{R}^E$ and $\boldsymbol{\omega}\in(0,\infty)^V$ such that  $\Sigma=\phi_G(\Lambda,\boldsymbol{\omega})\in M_G$.  But then, by Theorem~\ref{wii_thm}, the equalities among conditional variances imply that $\omega_{ii}=\omega_{jj}$ for $i\sim_\Pi j$.  Therefore, $\Sigma\in M_{G,\Pi}$.
\end{proof}

\section{Equivalence classes}\label{sec_eq}

\subsection{Model equivalence of DAGs}

Let $G_1=(V,E_1)$ and $G_2=(V,E_2)$ be two DAGs with the same given vertex set.  An important problem for causal discovery is to decide whether the two DAGs are equivalent in the sense of defining the same statistical model for the observations at hand.

\begin{definition}
\label{def:model:equivalence}
Let $\Pi$ be a fixed partition of the index set $V$.  Two DAGs $G_1=(V,E_1)$ and $G_2=(V,E_2)$ are \emph{$\Pi$-model equivalent} if $M_{G_1,\Pi}=M_{G_2,\Pi}$.  In this case, we write $G_1\approx_\Pi G_2$.
\end{definition}

The classic case of linear SEMs with arbitrary error variances corresponds to the partition $\Pi_{\min}=\{\{i\}:i\in V\}$.  In this case we have no equal variance constraints, and $G_1$ and $G_2$ are $\Pi_{\min}$-model equivalent if and only if $G_1$ and $G_2$ are Markov equivalent, meaning they induce the same conditional independence relations or, equivalently, have the same $d$-separation relations.  Graphical models theory further tells us that $G_1$ and $G_2$ are Markov equivalent if and only if they have the same skeleton and unshielded collider triples \citep{handbook:studeny}.

Based on the algebraic characterization in Theorem~\ref{ideal}, we are able to give the following extension of the Markov equivalence theory to the setting of partially homoscedastic models.

\begin{theorem}\label{equi_cl}
Let $G_1=(V,E_1)$ and $G_2=(V,E_2)$ be two DAGs, and let $\Pi=\{\pi_1,\dots,\pi_K\}$ be a partition of the index set $V$.  Then $G_1$ and $G_2$ are $\Pi$-model equivalent if and only if the following two conditions hold:
\begin{enumerate}[label=(\roman*)]
\item $G_1$ and $G_2$ have the same skeleton and unshielded colliders, and
\item $\pa_{G_1}(i)=\pa_{G_2}(i)$ for all nodes $i$ that belong to a partition block $\pi_k$ of size $|\pi_k|\geq 2$.
\end{enumerate}
\end{theorem}

\begin{proof}
For the ``if'' direction, suppose that conditions (i) and (ii) hold.  By the standard Markov equivalence theory, condition (i) implies that $G_1$ and $G_2$ have the same $d$-separation relations and, thus, $M_{G_1}=M_{G_2}$.   Now, let $\Sigma$ be an arbitrary element of $M_{G_1,\Pi}$.  Since $M_{G_1,\Pi}\subseteq M_{G_1}=M_{G_2}$, there is a (unique) choice of $\Lambda^{(2)}\in\mathbb{R}^{E_2}$ and $\boldsymbol{\omega}^{(2)}\in(0,\infty)^V$ such that $\Sigma=\phi_{G_2}(\Lambda^{(2)},\boldsymbol{\omega}^{(2)})$.  Let $i\not=j$ be any two nodes with $i\sim_\Pi j$, i.e., there is a partition block $\pi_k$ of size $|\pi_k|\ge 2$ that contains both $i,j$.   By Corollary~\ref{wij_eq_col}, since $\Sigma\in M_{G_1,\Pi}$, we have
\begin{align*}
    \sigma_{ii}-\Sigma_{i,\pa_{G_1}(i)}(\Sigma_{\pa_{G_1}(i),\pa_{G_1}(i)})^{-1}\Sigma_{\pa_{G_1}(i),i}=\sigma_{jj}-\Sigma_{j,\pa_{G_1}(j)}(\Sigma_{\pa_{G_1}(j),\pa_{G_1}(j)})^{-1}\Sigma_{\pa_{G_1}(j),j}.
\end{align*}
By condition (ii), $\pa_{G_1}(i)=\pa_{G_2}(i)$ and $\pa_{G_1}(j)=\pa_{G_2}(j)$.  Therefore, we have
\begin{align*}
  \omega^{(2)}_{ii} &= \sigma_{ii}-\Sigma_{i,\pa_{G_2}(i)}(\Sigma_{\pa_{G_2}(i),\pa_{G_2}(i)})^{-1}\Sigma_{\pa_{G_2}(i),i}\\
  &=\sigma_{jj}-\Sigma_{j,\pa_{G_2}(j)}(\Sigma_{\pa_{G_2}(j),\pa_{G_2}(j)})^{-1}\Sigma_{\pa_{G_2}(j),j} =\omega^{(2)}_{jj}.
\end{align*}
We conclude that $\Sigma\in M_{G_2,\Pi}$ and, thus, $M_{G_1,\Pi}\subseteq M_{G_2,\Pi}$.  Swapping the role of $G_1$ and $G_2$, we conclude that $M_{G_1,\Pi}= M_{G_2,\Pi}$ and $G_1\approx_\Pi G_2$.

For the ``only if'' direction, suppose $M_{G_1,\Pi}= M_{G_2,\Pi}$.
Theorem~\ref{ideal} implies that $G_1$ and $G_2$ induce the same conditional independence constraints and the same set of equal variance constraints (as specified in Corollary~\ref{wij_eq_col}). We deduce that $G_1$ and $G_2$ have the same $d$-separation relations and, thus, condition (i) holds.  Let $i,j$ be any two distinct nodes in the same partition block $\pi_k$.  Since $G_1$ and $G_2$ induce the same set of equal variance constraints, the set $\mathcal{A}_{ij}$ defined in Corollary~\ref{cor:Aij} is the same for $G_1$ as for $G_2$.  Corollary~\ref{cor:Aij} now implies that the unique minimal element of $\mathcal{A}_{ij}$ must be comprised of the parent sets of node $i$ and $j$ in both $G_1$ and $G_2$.  But this means that $\pa_{G_1}(i)=\pa_{G_2}(i)$ and $\pa_{G_1}(j)=\pa_{G_2}(j)$.  Therefore, condition (ii) holds.
\end{proof}

\begin{remark}
The two extreme cases of our setup are the classic setting in which all variances are freely varying ($|\Pi|=|V|$ or in other words $\Pi=\Pi_{\min}=\{\{i\}:i\in V\}$) and the previously studied case with all variances equal ($|\Pi|=1$ or in other words $\Pi=\Pi_{\max}=\{V\}$).  When $\Pi=\Pi_{\min}$, condition (ii) in  Theorem~\ref{equi_cl} never applies and the theorem is just the classic Markov equivalence theorem.  When $\Pi=\Pi_{\max}$, condition (ii) applies to all nodes, and Theorem~\ref{equi_cl} thus recovers the fact that under an equal variance assumption no two DAGs define the same model.
\end{remark}

\begin{remark}
Another interesting special case arises in the context of two-sample problems, in which we observe each one of $d$ variables under two different experimental conditions.  In this setting, it is of interest to estimate the difference between the two DAGs for the two samples.  This problem is greatly simplified by assuming equality of the two error variances that arise in the structural equations for the two independent copies of the $k$th random variables, $k=1,\dots,d$ \citep{dif}.  We can accommodate the two-sample problem in our framework but grouping all $2d$ random variables together.  We then have a single combined graph of even size $|V|=2d$ that joins the two DAGs for the two samples.  The equal variance assumption in \citet{dif} corresponds to a partition $\Pi=\{\pi_1,\dots,\pi_d\}$ with $|\pi_1|=\cdots=|\pi_d|=2$.  Theorem \ref{equi_cl} implies that under this partition the combined DAG is uniquely determined by the joint distribution of the observations from the two samples.
\end{remark}

\subsection{Completed partially directed acyclic graph (CPDAG)}
Beyond characterizing equivalence of two DAGs as in Theorem \ref{equi_cl}, it is also of interest to provide a representation of each equivalence class.
Similar to the classic heteroscedastic setup, we can represent the equivalence class by a \emph{completed partially directed acyclic graph} (CPDAG) \citep{andersson1997characterization}.

\begin{definition}
Let $\Pi$ be a partition of the vertex set of a DAG $G=(V,E)$.
The \emph{completed partially directed acyclic graph} (CPDAG) of the DAG $G$ \emph{under partition $\Pi$} is the graph obtained by forming the union of all DAGs equivalent to $G$:
\begin{align}
G^*_\Pi := \cup\ (G'\mid G'\approx_\Pi G).
\end{align}
So, $G^*_\Pi$ contains edge $i\to j$ if the edge is contained in  some DAG $G'\approx_\Pi G$.  It is customary to draw $G^*_\Pi$ as a mixed graph with an undirected edge between nodes $i$ and $j$ for which both $i\to j$ and $j\to i$ are in $G^*_\Pi$.  
\end{definition}

We emphasize that an undirected edge in a CPDAG indicates that there exist two DAGs in the equivalence class in which the edge appears with opposite directions.  Moreover, a CPDAG contains a directed edge $i\to j$ precisely when all DAGs in the equivalence class of $G$ contain this edge. 

In the standard heteroscedastic setting (i.e., $\Pi=\Pi_{\min}=\{\{i\}:i\in V\}$), the CPDAG may be constructed using an algorithm of  \citet{meek1995causal}.  In addition, \citet{meek1995causal} shows how to construct a CPDAG in a setting where there is background knowledge about some of the edges.
 The background knowledge is of the form $
\mathcal{K}=\langle\textbf{F},\textbf{R}\rangle$, where $\textbf{F}$ contains the edges not in the DAG and $\textbf{R}$ contains the edges in the DAG. The algorithm first translates conditional independence statements into adjacencies and unshielded collider triples. Then the first 3 of the 4 orientation rules in \citet{VP92} (Figure \ref{ORs}) are applied to obtain the CPDAG without background knowledge, which is exactly the CPDAG under $\Pi_p$. The last phase incorporates background knowledge and checks whether a compatible CPDAG exists or not.  The following is the procedure, in which background knowledge is inserted edge by edge, and the CPDAG at the current step is denoted by $G^*$:
\begin{itemize}
\item [S1] If there is an edge $i\rightarrow j$ in $\textbf{F}$ such that $i\rightarrow j$ in $G^*$ then FAIL.
\item [S1$'$] If there is an edge $i\rightarrow j$ in $\textbf{R}$ such that $j\rightarrow i$ in $G^*$ or $i,j$ are not adjacent in $G^*$ then FAIL.
\item [S2] Randomly choose one edge $i\rightarrow j$ from $\textbf{R}$, and let $\textbf{R}=\textbf{R}\backslash\{i\rightarrow j\}$.
\item [S3] Orient $i\rightarrow j$ in $G^*$ and close orientations under rules R1, R2, R3 and R4 in Figure \ref{ORs}.
\item [S4] If $\textbf{R}\not=\emptyset$, then go to step S1.
\end{itemize}

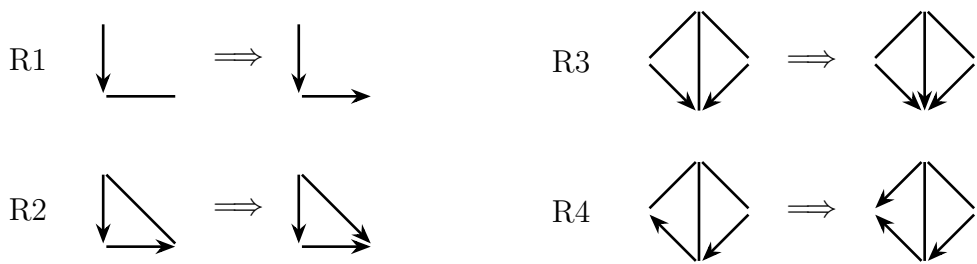
\begin{figure}[htbp]
\centering
\begin{minipage}[c]{0.48\textwidth}
\centering
\begin{tikzpicture}
[main_node/.style={draw=none,fill=none,inner sep=1pt}]
\node[main_node] (1) at (0, 4) {R1};
\node[main_node] (2) at (1, 4.5) {};
\node[main_node] (3) at (1, 3.5) {};
\node[main_node] (4) at (2, 3.5) {};
\node[main_node] (5) at (2.8, 4) {$\implies$};
\node[main_node] (6) at (3.6, 4.5) {};
\node[main_node] (7) at (3.6, 3.5) {};
\node[main_node] (8) at (4.6, 3.5) {};

\node[main_node] (9) at (0, 2) {R2};
\node[main_node] (10) at (1, 2.5) {};
\node[main_node] (11) at (1, 1.5) {};
\node[main_node] (12) at (2, 1.5) {};
\node[main_node] (13) at (2.8, 2) {$\implies$};
\node[main_node] (14) at (3.6, 2.5) {};
\node[main_node] (15) at (3.6, 1.5) {};
\node[main_node] (16) at (4.6, 1.5) {};

\draw [-{Stealth}, line width=1pt] (2) to (3);
\draw [-, line width=1pt] (3) to (4);
\draw [-{Stealth}, line width=1pt] (6) to (7);
\draw [-{Stealth}, line width=1pt] (7) to (8);

\draw [-{Stealth}, line width=1pt] (10) to (11);
\draw [-, line width=1pt] (10) to (12);
\draw [-{Stealth}, line width=1pt] (11) to (12);
\draw [-{Stealth}, line width=1pt] (14) to (15);
\draw [-{Stealth}, line width=1pt] (15) to (16);
\draw [-{Stealth}, line width=1pt] (14) to (16);

\end{tikzpicture}
\end{minipage}
\centering 
\begin{minipage}[c]{0.48\textwidth}
\centering
\begin{tikzpicture}
[main_node/.style={draw=none,fill=none,inner sep=1pt}]
\node[main_node] (1) at (0, 4) {R3};
\node[main_node] (2) at (1, 4) {};
\node[main_node] (3) at (1.7, 4.7) {};
\node[main_node] (4) at (2.4, 4) {};
\node[main_node] (5) at (1.7, 3.3) {};

\node[main_node] (6) at (3.2, 4) {$\implies$};
\node[main_node] (7) at (4, 4) {};
\node[main_node] (8) at (4.7, 4.7) {};
\node[main_node] (9) at (5.4, 4) {};
\node[main_node] (10) at (4.7, 3.3) {};

\node[main_node] (11) at (0, 2) {R4};
\node[main_node] (12) at (1, 2) {};
\node[main_node] (13) at (1.7, 2.7) {};
\node[main_node] (14) at (2.4, 2) {};
\node[main_node] (15) at (1.7, 1.3) {};

\node[main_node] (16) at (3.2, 2) {$\implies$};
\node[main_node] (17) at (4, 2) {};
\node[main_node] (18) at (4.7, 2.7) {};
\node[main_node] (19) at (5.4, 2) {};
\node[main_node] (20) at (4.7, 1.3) {};

\draw [-, line width=1pt] (2) to (3);
\draw [-{Stealth}, line width=1pt] (2) to (5);
\draw [-, line width=1pt] (3) to (4);
\draw [-{Stealth}, line width=1pt] (4) to (5);
\draw [-, line width=1pt] (3) to (5);

\draw [-, line width=1pt] (7) to (8);
\draw [-{Stealth}, line width=1pt] (7) to (10);
\draw [-, line width=1pt] (8) to (9);
\draw [-{Stealth}, line width=1pt] (9) to (10);
\draw [-{Stealth}, line width=1pt] (8) to (10);

\draw [-, line width=1pt] (12) to (13);
\draw [-{Stealth}, line width=1pt] (15) to (12);
\draw [-, line width=1pt] (13) to (14);
\draw [-{Stealth}, line width=1pt] (14) to (15);
\draw [-, line width=1pt] (13) to (15);

\draw [-{Stealth}, line width=1pt] (18) to (17);
\draw [-{Stealth}, line width=1pt] (20) to (17);
\draw [-, line width=1pt] (18) to (19);
\draw [-{Stealth}, line width=1pt] (19) to (20);
\draw [-, line width=1pt] (18) to (20);

\end{tikzpicture}
\end{minipage}
\caption{The four orientation rules.}
\label{ORs}
\end{figure}

In our setup, we want to compute the equivalence class 
of a DAG under a partition, which restricts edge orientations in the DAG's equivalence class.  These restrictions can be interpreted as providing background knowledge as considered by Meek. In our case, a CPDAG compatible to the background knowledge always exists, and we can use a simplified version of the general algorithm to construct the equivalence class. 

Given a DAG $G$ and a partition $\Pi$, the equivalence class is obtained by the following algorithm (Algorithm \ref{alg:1}). Theorem \ref{alg_thm} below certifies the correctness of the algorithm. The proof is deferred to Appendix \ref{A3}.

\begin{algorithm}
\caption{Constructing the equivalence class of a DAG, given the partition}\label{alg:1}
\begin{algorithmic}[1]
\Require A DAG $G$, the partition $\Pi$
\State Create an empty graph $G'$
\State Copy the skeleton and all edge orientations with unshielded colliders of $G$ to $G'$
\State Apply rules R1, R2 and R3 on $G'$ until no more edges can be oriented
\For{$i\in V$ with $i\in\pi_k$ and $|\pi_k|\geq 2$}
\State Copy the orientation of edges in $G$ having one endpoint at $i$ to $G'$
\EndFor
\State Apply rules R1 and R2 on $G'$ until no more edges can be oriented\\
\Return $G^*_\Pi=G'$
\end{algorithmic}
\end{algorithm}

\begin{theorem}\label{alg_thm}
Given a DAG $G$ and partition $\Pi$, Algorithm \ref{alg:1} outputs the CPDAG $G^*_\Pi$. 
\end{theorem}

\begin{example}
Consider the DAG $G$ in Figure \ref{CPDAGeg} with node set $V=\{1,2,3,4,5,6\}$ and the partition $\Pi=\{\{1,2\},\{3\},\{4\},\{5\},\{6\}\}$. In other words, the partition sequence is $(1,1,2,3,4,5)$, where the $i'$th element of the sequence indicates the block that node $i$ belongs to. To determine the equivalence class of $G$, we first keep the skeleton and unshielded colliders. Then those edges containing node $1$ or $2$ (partition block size $\geq 2$) are oriented the same way as they are in $G$. Next, we propagate the edge orientation by rules R1 and R2, and we find that the edge between $4$ and $6$ is oriented as $4\rightarrow 6$. Finally, the remaining edge $3-5$ can have both direction and is kept undirected in the final CPDAG that represents the equivalence class of $G$.

\begin{figure}
\centering
\begin{minipage}[c]{0.48\textwidth}
\centering
\begin{tikzpicture}
[node1/.style={circle,fill=white,draw=brown,minimum size=2 em,inner sep=2pt, line width=1pt},
node2/.append style={circle,fill=white,draw=blue,minimum size=2 em,inner sep=2pt, line width=1pt},
node3/.append style={circle,fill=white,draw=purple,minimum size=2 em,inner sep=2pt, line width=1pt},
node4/.append style={circle,fill=white,draw=green,minimum size=2 em,inner sep=2pt, line width=1pt},
node5/.append style={circle,fill=white,draw=yellow,minimum size=2 em,inner sep=2pt, line width=1pt}
]
\node[node1] (1) at (-1, 1.732) {$1$};
\node[node1] (2) at (1, 1.732) {$2$};
\node[node2] (3) at (-2, 0) {$3$};
\node[node3] (4) at (2, 0) {$4$};
\node[node4] (5) at (-1, -1.732) {$5$};
\node[node5] (6) at (1, -1.732) {$6$};

\draw [->, line width=1pt] (1) to (2);
\draw [->, line width=1pt] (1) to (3);
\draw [->, line width=1pt] (2) to (3);
\draw [->, line width=1pt] (1) to (5);
\draw [->, line width=1pt] (2) to (5);
\draw [->, line width=1pt] (3) to (5);
\draw [->, line width=1pt] (2) to (6);
\draw [->, line width=1pt] (6) to (4);

\end{tikzpicture}
\subcaption{DAG $G$}
\end{minipage}
\begin{minipage}[c]{0.48\textwidth}
\centering
\begin{tikzpicture}
[node1/.style={circle,fill=white,draw=brown,minimum size=2 em,inner sep=2pt, line width=1pt},
node2/.append style={circle,fill=white,draw=blue,minimum size=2 em,inner sep=2pt, line width=1pt},
node3/.append style={circle,fill=white,draw=purple,minimum size=2 em,inner sep=2pt, line width=1pt},
node4/.append style={circle,fill=white,draw=green,minimum size=2 em,inner sep=2pt, line width=1pt},
node5/.append style={circle,fill=white,draw=yellow,minimum size=2 em,inner sep=2pt, line width=1pt}
]
\node[node1] (1) at (-1, 1.732) {$1$};
\node[node1] (2) at (1, 1.732) {$2$};
\node[node2] (3) at (-2, 0) {$3$};
\node[node3] (4) at (2, 0) {$4$};
\node[node4] (5) at (-1, -1.732) {$5$};
\node[node5] (6) at (1, -1.732) {$6$};

\draw [->, line width=1pt] (1) to (2);
\draw [->, line width=1pt] (1) to (3);
\draw [->, line width=1pt] (2) to (3);
\draw [->, line width=1pt] (1) to (5);
\draw [->, line width=1pt] (2) to (5);
\draw [-, line width=1pt, color=red] (3) to (5);
\draw [->, line width=1pt] (2) to (6);
\draw [->, line width=1pt, color=red] (6) to (4);

\end{tikzpicture}
\subcaption{CPDAG $G'$}
\end{minipage}
\caption{A DAG and the corresponding CPDAG, under a fixed partition.}
\label{CPDAGeg}
\end{figure}
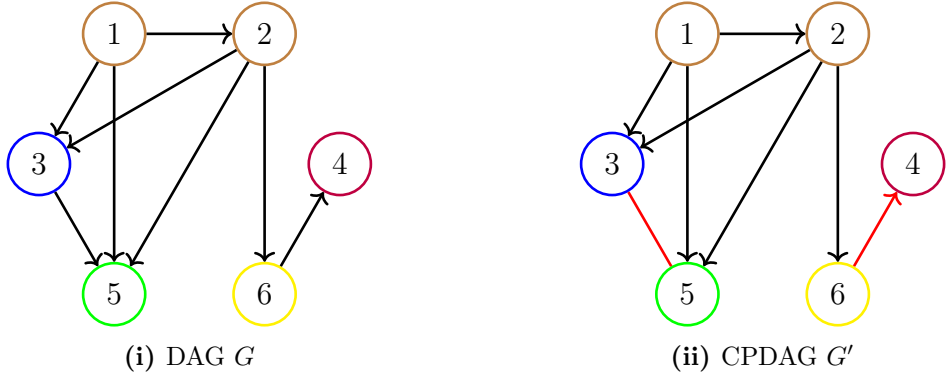

\end{example}

\section{Greedy search}\label{sec_gs}
\subsection{Likelihood inference}

Let $\mathbf{X}=(X_1,\dots, X_p)^T\in\mathbb{R}^{p\times n}$ be a data matrix comprised of $n$ observations for each one of the $|V|=p$ considered variables.  The columns of $\mathbf{X}$ are assumed to be generated as an i.i.d.~sample from a joint multivariate normal distribution.  Without loss of generality, we may assume the mean vector of the normal distribution to be zero. (Otherwise, we may estimate the means by sample means and recenter each row of the data matrix.) Define the sample covariance matrix $S=\mathbf{X}\mathbf{X}^T/n$.
For a fixed DAG $G=(V,E)$ and partition $\Pi$ of $V$, the partially homoscedastic linear Gaussian model given by $(G,\Pi)$ has log-likelihood function
\begin{align}
\ell_{G}(\Lambda,\boldsymbol{\omega})&=\frac{n}{2}\left(-\log \det(\diag(\boldsymbol{\omega}))+\log \det(I-\Lambda)^{2} -\tr\left\{(I-\Lambda) \diag(\boldsymbol{\omega})^{-1}(I-\Lambda)^T S\right\}\right).\nonumber 
\end{align}
In order to compute the maximum likelihood estimates that maximize the function, one may solve a sequence of linear regression problems, one problem for each node/variable with predictors given by parent variables.

Let $\Pi=\{\pi_1,\dots,\pi_K\}$.  Then 
 \begin{align}
\ell_{G}(\Lambda,\boldsymbol{\omega})
&=
\frac{n}{2}\sum_{k=1}^K\left(-|\pi_k|\log\omega_k-\frac{1}{n\omega_k}\left(\sum_{i\in\pi_k}\left\|X_i-\Lambda_{i,\pa(i)}^T X_{\pa(i)}\right\|^2\right)\right)\nonumber \\
&:=\frac{n}{2}\sum_{k=1}^K \ell_{G,\pi_k}(\Lambda,\omega_k),
\end{align}
which is the sum of log-likelihood values of the $K$ blocks.  We deduce that the maximum log-likelihood is achieved at the pair $(\widehat{\Lambda},\widehat{\boldsymbol{\omega}})$ with 
\begin{align*}
\widehat{\Lambda}_{\pa(i),i}&=\argmin_{\beta\in\RR^{|\pa(i)|}}\|X_i-\beta^T X_{\pa(i)}\|^2,\\
\widehat{\omega}_k&=\frac{\sum_{i\in\pi_k}\left\|X_i-\widehat{\Lambda}_{\pa(i),i}^TX_{\pa(i)}\right\|^2}{n|\pi_k|}.
\end{align*}

In order to solve the problem of selecting the DAG $G$, we appeal to information criteria.  Plugging the maximum likelihood estimate $(\widehat{\Lambda},\widehat{\boldsymbol{\omega}})$ back into the log-likelihood function, we can compute the Bayesian information criterion (BIC) score for the DAG $G$ given the data $\mathbf{X}$.  This score decomposes into the sum of scores of each block:
\begin{align}\label{bic}
s_{\text{BIC}}(G)&=\frac{1}{n}\left(\ell_{G}(\widehat{\Lambda},\widehat{\boldsymbol{\omega}})-\frac{\log(n)}{2}|E|\right)\nonumber\\
&=\frac{1}{2}\sum_{k=1}^K\left(-|\pi_k|\log\hat{\omega}_k-|\pi_k|-\frac{\log(n)}{n}\sum_{i\in\pi_k}|\pa(i)|\right).
\end{align}

\subsection{Search scheme}\label{search}
For model selection, we maximize the BIC score over the space of DAGs.  The number of possible DAGs grows very quickly with the number of nodes $p$; e.g., we have
$1.2\times 10^{15}$ for $p=10$ \citep{oeis}.
We thus follow
prior work and adopt a greedy search algorithm, which starts at some initial random or empty DAG and selects the DAG with highest BIC score in the local neighborhood at each step.  The procedure terminates when the considered DAG has higher BIC score than all other DAGs in the local neighborhood.  Here, we define the local neighborhood of a DAG $G$ as the collection of all DAGs that can be obtained from $G$ by one edge addition, removal or reversal. 

An edge addition or removal always changes the equivalence class of a DAG.
Whether an edge reversal creates a DAG in a different equivalence class is determined by parents and partition information, as specified in our previous results.  We can thus search DAGs that are in the neighborhood and in different equivalence classes. 
To speed up the search we also restrict the local neighborhood to a random subset with size bound.
To relieve issues of local optima, we perform the greedy search multiple times starting from different initial DAGs. In this work, for each realization of the greedy search, we restart the method 5 times with neighborhood size bound $k=300$.  We refer to this scheme as GEV.

\subsection{Simulation study}\label{sec_sim}
We investigate the numerical performance of our algorithm (GEV) by comparing against the greedy equivalence search (GES) \citep{chickering02b} and the PC-algorithm \citep{spirtes:2000}. The former tries to find the structure with maximum $l_0$-penalized log-likelihood and the default penalty is $\log(n)/2n$, corresponding to the BIC score. The latter has a significance level $\alpha$ for conditional independence tests that determine edges. To make the score-based and the constraint-based methods comparable, we consider a grid of values for $\alpha$ from $10^{-5}$ to $0.8$, increasing by the ratio $1.1$ \citep{PCnonpara}. Then we can choose the value of $\alpha$ according to the maximum BIC score. 

Both PC and GES algorithm do not incorporate any possible (partial) homoscedasticity and return a standard Markov equivalence class. In our partially homoscedastic model setup, the greedy search is performed among DAGs and returns the CPDAG of the final DAG, as described in Section \ref{search}. Since the parental information (or edge directions) plays an important role in our idenfitiability result, we adopt the modified structural Hamming distance (SHD) in \citet{Peters2014biometrika} as the error measurement. The classic SHD is the number of edge additions, deletions and reversals in order to transform one graph to another graph, i.e., all edge mistakes count as 1, while the modified version assigns a distance of 2 to each pair of reversed edges.

In the simulation we use 24 different configurations of $(p,n,sp)$, where $p\in\{5,10,20,40\}$ is the number of nodes, $n\in\{100,500,1000\}$ is the sample size and $sp\in\{\textit{sparse}, \textit{dense}\}$ controls the sparsity of randomly generated DAGs. In the sparse setting, each pair of nodes has the adjacency probability $\textit{prob}=3/(2p-2)$, while in the dense setting the probability is $0.3$. For each $(i,j)$ pair with $i<j$, we simulate independent uniform random variables $U_{ij}\sim U(0,1)$. If $U_{ij}<\textit{prob}$, the edge $i\rightarrow j$ is introduced. Every edge weight is uniformly drawn from $[-1,-0.3]\cup [0.3,1]$,
and the error variance of each partition block is uniformly drawn from $[0.3,1]$. After traversing all node pairs, we randomly permute the node labels. For each configuration we run the simulation 100 times.

The box-plots that follow show the SHD between the true CPDAG and the estimated CPDAG obtained by the considered methods.  We study the case of 2  partition blocks as well as a more subtle case with $\lceil p/3\rceil+1$ blocks.   In the former case, the simulation experiments summarized in Figures \ref{fig:box-p/3-s} and \ref{fig:box-p/3-d} show that the greedy search algorithm for partially homoscedastic models is able to very effectively exploit the available homoscedasticity in both sparse and dense settings as well as across all three sample sizes and four dimensions.  Its SHDs are consistently lower as for GES and PC.  This said, the dense is clearly far more challenging than the sparse---as is to be expected.  Moreover, the simulations confirm the intuition that the SHDs should be smaller if extra equal error variances information is utilized by the algorithm.  For the larger number of blocks, Figure~\ref{fig:box-2-s} shows the same pattern of clearly better performance.  However, in the dense case depicted in Figure~\ref{fig:box-2-d}) one now sees the problem becoming difficult in the highest-dimensional case where the PC algorithm shows best performance.

\begin{figure}[tbp] 
\center
%\hspace*{-.6cm}
\includegraphics[scale=0.42]{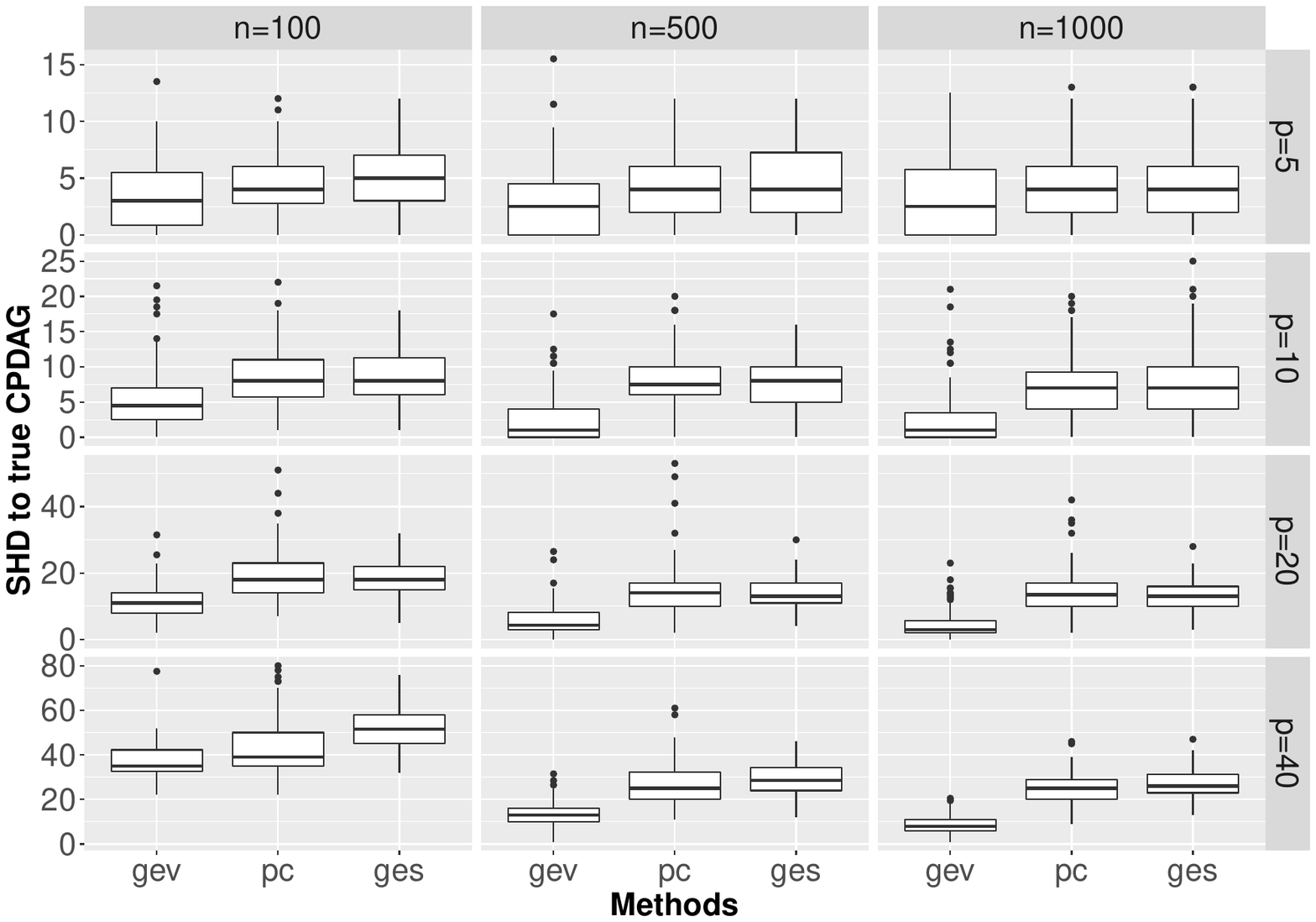}
\caption{Box-plots of SHD by groups of $p$ and $n$, sparse graphs $2$ blocks.}
\label{fig:box-2-s}
\end{figure}

\begin{figure}[tbp] 
\center
%\hspace*{-.6cm}
\includegraphics[scale=0.42]{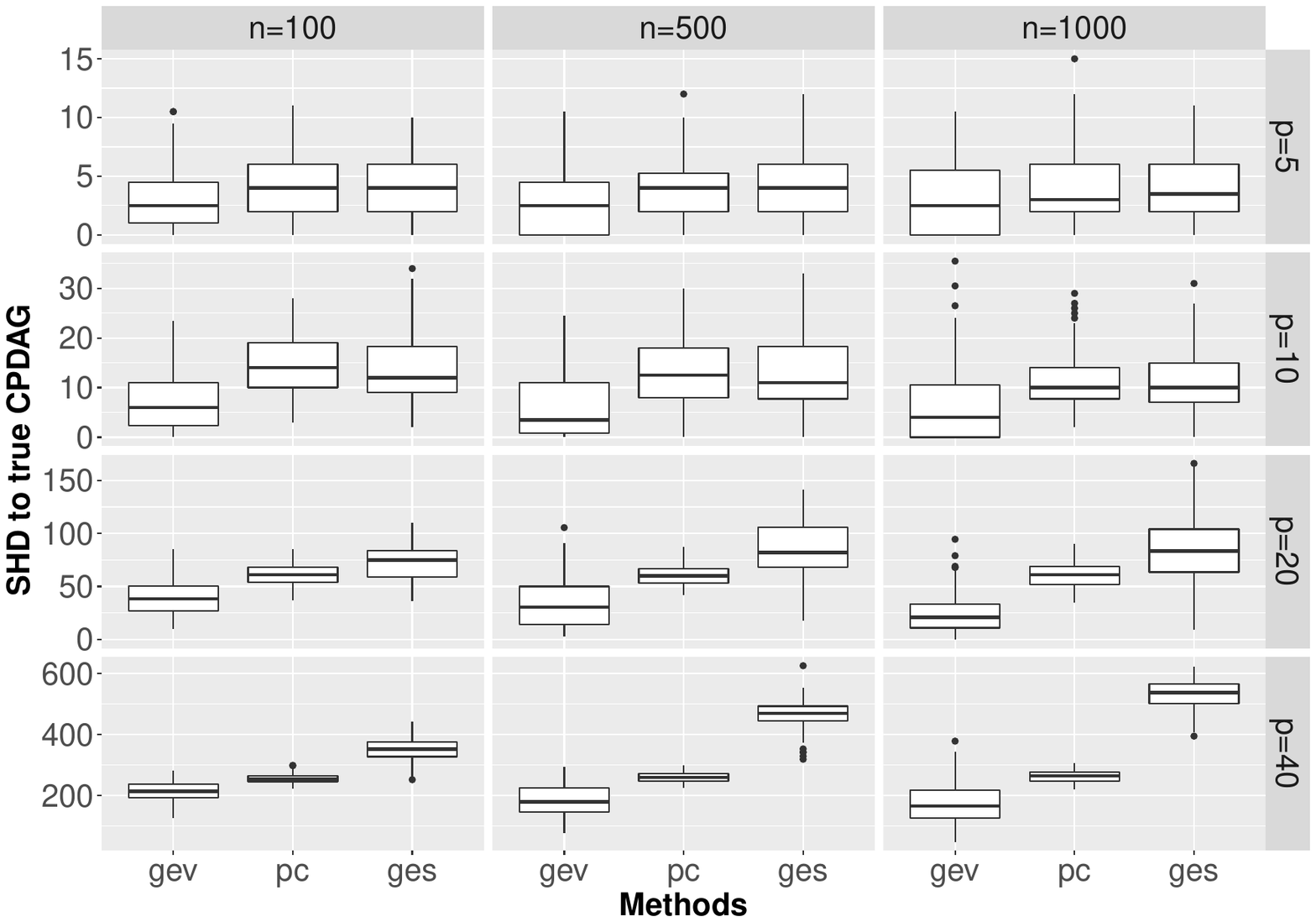}
\caption{Box-plots of SHD by groups of $p$ and $n$, dense graphs, $2$ blocks.}
\label{fig:box-2-d}
\end{figure}

\begin{figure}[tbp] 
\center
%\hspace*{-.6cm}
\includegraphics[scale=0.42]{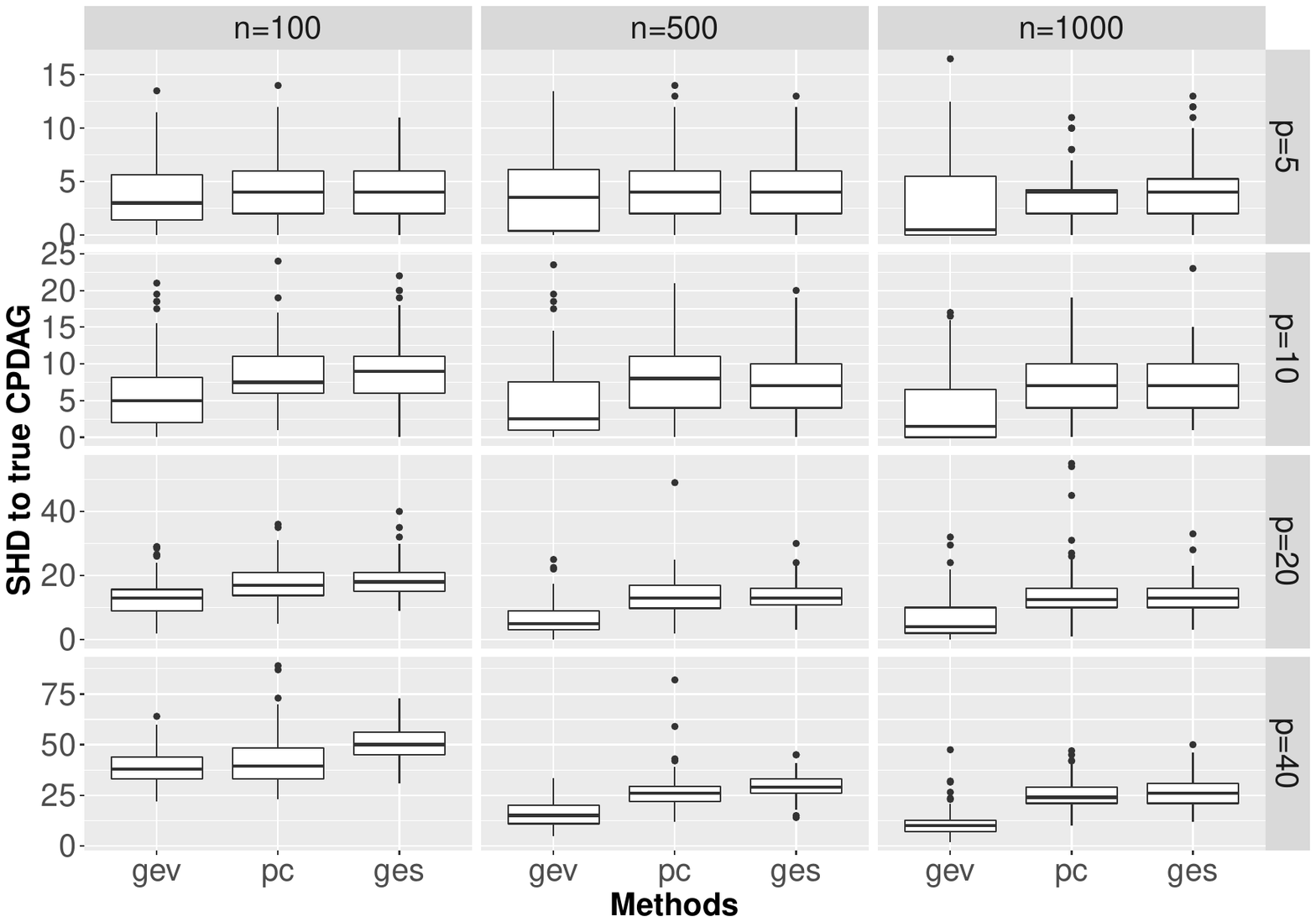}
\caption{Box-plots of SHD by groups of $p$ and $n$, sparse graphs, $\lceil p/3 \rceil +1$ blocks.}
\label{fig:box-p/3-s}
\end{figure}

\begin{figure}[tbp] 
\center
%\hspace*{-.6cm}
\includegraphics[scale=0.42]{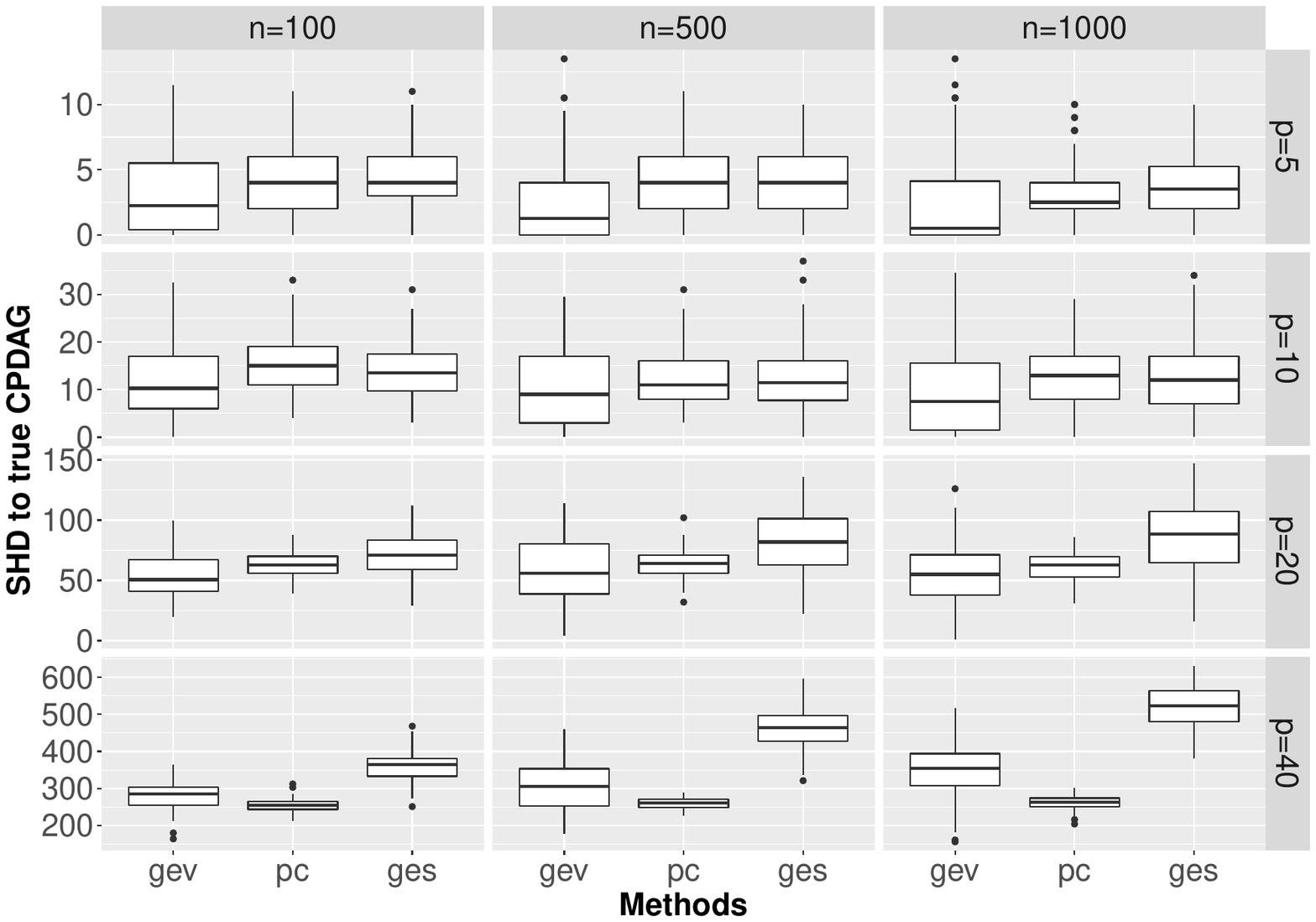}
\caption{Box-plots of SHD by groups of $p$ and $n$, dense graphs, $\lceil p/3 \rceil +1$ blocks.}
\label{fig:box-p/3-d}
\end{figure}

\section{Discussion}
\label{sec_discuss}

The framework of partially homoscedastic linear Gaussian models is a generalization of linear SEMs with equal error variances.  It encodes equal variance assumption through a partition of the variables.  The framework  unifies the classical setting in which the error variances may be arbitrary and the  equal error variance setup that has been studied in recent literature.  These two cases are captured by the two extreme partitions,  with a single block and all variables in separate blocks, respectively. 

Each partially homoscedastic linear model can be characterized algebraically via conditional independence constraints and equal variance constraints.  The former are well known from the classical graphical model perspective on linear SEMs, and we explicitly derived the latter in this paper. The equal variance constraints  reveal the essence of how equal variance assumptions lead to identifiability of edge orientations.  This perspective differs from previous work on the equal variance assumption which, e.g., considered ordering variances \citep[e.g.][]{chen2018causal}.
We also showed how equivalance classes in the partially homoscedastic setting are naturally represented by a refined CPDAG, which may be constructed efficiently with the help of existing results on CPDAGs in setting with background knowledge.
For model selection, we demonstrated that greedy search provides an effective tool to exploit knowledge about partial homoscedasticity.

\iffund \section*{Acknowledgements} This project has received funding from the European Research Council (ERC) under the European Union’s Horizon 2020 research and innovation programme (grant agreement No 883818). \fi

\section*{References}
\renewcommand{\bibsection}{}
\bibliography{ref}

\begin{thebibliography}{28}
\providecommand{\natexlab}[1]{#1}
\providecommand{\url}[1]{\texttt{#1}}
\expandafter\ifx\csname urlstyle\endcsname\relax
  \providecommand{\doi}[1]{doi: #1}\else
  \providecommand{\doi}{doi: \begingroup \urlstyle{rm}\Url}\fi

\bibitem[Andersson et~al.(1997)Andersson, Madigan, and
  Perlman]{andersson1997characterization}
Steen~A. Andersson, David Madigan, and Michael~D. Perlman.
\newblock A characterization of {M}arkov equivalence classes for acyclic
  digraphs.
\newblock \emph{The Annals of Statistics}, 25\penalty0 (2):\penalty0 505--541,
  1997.

\bibitem[Chen et~al.(2019)Chen, Drton, and Wang]{chen2018causal}
Wenyu Chen, Mathias Drton, and Y.~Samuel Wang.
\newblock On causal discovery with an equal-variance assumption.
\newblock \emph{Biometrika}, 106\penalty0 (4):\penalty0 973--980, 2019.

\bibitem[Chickering(2003)]{chickering02b}
David~Maxwell Chickering.
\newblock Optimal structure identification with greedy search.
\newblock \emph{Journal of Machine Learning Research}, 3\penalty0 (3):\penalty0
  507--554, 2003.
\newblock Computational learning theory.

\bibitem[Drton(2018)]{drton:algebraic}
Mathias Drton.
\newblock Algebraic problems in structural equation modeling.
\newblock In \emph{The 50th anniversary of {G}r\"{o}bner bases}, volume~77 of
  \emph{Advanced Studies in Pure Mathematics}, pages 35--86. Mathematical
  Society of Japan, 2018.

\bibitem[Drton and Maathuis(2017)]{drton:maathuis:2017}
Mathias Drton and Marloes~H. Maathuis.
\newblock Structure learning in graphical modeling.
\newblock \emph{Annu. Rev. Stat. Appl.}, 4:\penalty0 365--393, 2017.

\bibitem[Gao et~al.(2020)Gao, Ding, and Aragam]{gao2020polynomial}
Ming Gao, Yi~Ding, and Bryon Aragam.
\newblock A polynomial-time algorithm for learning nonparametric causal graphs.
\newblock \emph{Advances in Neural Information Processing Systems},
  33:\penalty0 11599--11611, 2020.

\bibitem[Gao et~al.(2022)Gao, Ming~Tai, and Aragam]{pmlr-v151-gao22a}
Ming Gao, Wai Ming~Tai, and Bryon Aragam.
\newblock Optimal estimation of {G}aussian {DAG} models.
\newblock In Gustau Camps-Valls, Francisco J.~R. Ruiz, and Isabel Valera,
  editors, \emph{Proceedings of The 25th International Conference on Artificial
  Intelligence and Statistics}, volume 151 of \emph{Proceedings of Machine
  Learning Research}, pages 8738--8757. PMLR, 28--30 Mar 2022.

\bibitem[Geiger and Pearl(1988)]{geiger1990logic}
Dan Geiger and Judea Pearl.
\newblock On the logic of causal models.
\newblock In Ross~D. Shachter, Tod~S. Levitt, Laveen~N. Kanal, and John~F.
  Lemmer, editors, \emph{{UAI} '88: Proceedings of the Fourth Annual Conference
  on Uncertainty in Artificial Intelligence, Minneapolis, MN, USA, July 10-12,
  1988}, pages 3--14. North-Holland, 1988.

\bibitem[Ghoshal and Honorio(2017)]{ghoshal2017learning}
Asish Ghoshal and Jean Honorio.
\newblock Learning identifiable {G}aussian {B}ayesian networks in polynomial
  time and sample complexity.
\newblock \emph{Advances in Neural Information Processing Systems}, 30, 2017.

\bibitem[Ghoshal and Honorio(2018)]{pmlr-v84-ghoshal18a}
Asish Ghoshal and Jean Honorio.
\newblock Learning linear structural equation models in polynomial time and
  sample complexity.
\newblock In Amos Storkey and Fernando Perez-Cruz, editors, \emph{Proceedings
  of the Twenty-First International Conference on Artificial Intelligence and
  Statistics}, volume~84 of \emph{Proceedings of Machine Learning Research},
  pages 1466--1475. PMLR, 2018.

\bibitem[Harris and Drton(2013)]{PCnonpara}
Naftali Harris and Mathias Drton.
\newblock P{C} algorithm for nonparanormal graphical models.
\newblock \emph{Journal of Machine Learning Research}, 14:\penalty0 3365--3383,
  2013.

\bibitem[Hoyer et~al.(2008)Hoyer, Janzing, Mooij, Peters, and
  Sch\"{o}lkopf]{NIPS2008_f7664060}
Patrik Hoyer, Dominik Janzing, Joris~M. Mooij, Jonas Peters, and Bernhard
  Sch\"{o}lkopf.
\newblock Nonlinear causal discovery with additive noise models.
\newblock In \emph{Advances in Neural Information Processing Systems},
  volume~21. Curran Associates, Inc., 2008.

\bibitem[Maathuis et~al.(2019)Maathuis, Drton, Lauritzen, and
  Wainwright]{handbook}
Marloes Maathuis, Mathias Drton, Steffen Lauritzen, and Martin Wainwright,
  editors.
\newblock \emph{Handbook of Graphical Models}.
\newblock Chapman \& Hall/CRC Handbooks of Modern Statistical Methods. CRC
  Press, Boca Raton, FL, 2019.

\bibitem[Meek(1995)]{meek1995causal}
Christopher Meek.
\newblock Causal inference and causal explanation with background knowledge.
\newblock In \emph{{UAI} '95: Proceedings of the Eleventh Annual Conference on
  Uncertainty in Artificial Intelligence, Montreal, Quebec, Canada, August
  18-20, 1995}, pages 403--410. Morgan Kaufmann, 1995.

\bibitem[Park(2020)]{park2020add}
Gunwoong Park.
\newblock Identifiability of additive noise models using conditional variances.
\newblock \emph{The Journal of Machine Learning Research}, 21\penalty0
  (1):\penalty0 2896--2929, 2020.

\bibitem[Park and Kim(2020)]{park2020gauss}
Gunwoong Park and Youngwhan Kim.
\newblock Identifiability of {G}aussian linear structural equation models with
  homogeneous and heterogeneous error variances.
\newblock \emph{Journal of the Korean Statistical Society}, 49:\penalty0
  276--292, 2020.

\bibitem[Pearl(2009)]{pearl:2009}
Judea Pearl.
\newblock \emph{Causality}.
\newblock Cambridge University Press, Cambridge, second edition, 2009.
\newblock Models, reasoning, and inference.

\bibitem[Peters and B\"{u}hlmann(2014)]{Peters2014biometrika}
Jonas Peters and Peter B\"{u}hlmann.
\newblock Identifiability of {G}aussian structural equation models with equal
  error variances.
\newblock \emph{Biometrika}, 101\penalty0 (1):\penalty0 219--228, 2014.

\bibitem[Peters et~al.(2011)Peters, Mooij, Janzing, and
  Sch{\"{o}}lkopf]{Peters11identifiabilityof}
Jonas Peters, Joris~M. Mooij, Dominik Janzing, and Bernhard Sch{\"{o}}lkopf.
\newblock Identifiability of causal graphs using functional models.
\newblock In \emph{{UAI} 2011, Proceedings of the Twenty-Seventh Conference on
  Uncertainty in Artificial Intelligence, Barcelona, Spain, July 14-17, 2011},
  pages 589--598. {AUAI} Press, 2011.

\bibitem[Richardson and Spirtes(2002)]{richardson:spirtes:2002}
Thomas Richardson and Peter Spirtes.
\newblock Ancestral graph {M}arkov models.
\newblock \emph{Ann. Statist.}, 30\penalty0 (4):\penalty0 962--1030, 2002.

\bibitem[Shimizu(2022)]{lingam:book}
Shohei Shimizu.
\newblock \emph{Statistical Causal Discovery: {LiNGAM} Approach}.
\newblock SpringerBriefs in Statistics. Springer Tokyo, 2022.

\bibitem[Shimizu et~al.(2006)Shimizu, Hoyer, Hyv\"{a}rinen, and
  Kerminen]{Shimizu06alinear}
Shohei Shimizu, Patrik~O. Hoyer, Aapo Hyv\"{a}rinen, and Antti Kerminen.
\newblock A linear non-{G}aussian acyclic model for causal discovery.
\newblock \emph{Journal of Machine Learning Research}, 7:\penalty0 2003--2030,
  2006.

\bibitem[Sloane and Inc.(2022)]{oeis}
Neil J.~A. Sloane and The OEIS~Foundation Inc.
\newblock The on-line encyclopedia of integer sequences, 2022.
\newblock \url{http://oeis.org/A003024}.

\bibitem[Spirtes et~al.(2000)Spirtes, Glymour, and Scheines]{spirtes:2000}
Peter Spirtes, Clark Glymour, and Richard Scheines.
\newblock \emph{Causation, Prediction, and Search}.
\newblock MIT Press, Cambridge, MA, second edition, 2000.

\bibitem[Studen\'{y}(2019)]{handbook:studeny}
Milan Studen\'{y}.
\newblock Conditional independence and basic {M}arkov properties.
\newblock In \emph{Handbook of Graphical Models}, Chapman \& Hall/CRC Handb.
  Mod. Stat. Methods, pages 3--38. CRC Press, Boca Raton, FL, 2019.

\bibitem[Verma and Pearl(1992)]{VP92}
Thomas Verma and Judea Pearl.
\newblock An algorithm for deciding if a set of observed independencies has a
  causal explanation.
\newblock In \emph{{UAI} '92: Proceedings of the Eighth Annual Conference on
  Uncertainty in Artificial Intelligence, Stanford University, Stanford, CA,
  USA, July 17-19, 1992}, pages 323--330. Morgan Kaufmann, 1992.

\bibitem[Wang and Bhattacharyya(2022)]{AMP}
Yuhao Wang and Arnab Bhattacharyya.
\newblock Identifiability of linear amp chain graph models.
\newblock \emph{Proceedings of the AAAI Conference on Artificial Intelligence},
  36\penalty0 (9):\penalty0 10080--10089, Jun. 2022.
\newblock \doi{10.1609/aaai.v36i9.21247}.
\newblock URL \url{https://ojs.aaai.org/index.php/AAAI/article/view/21247}.

\bibitem[Wang et~al.(2018)Wang, Squires, Belyaeva, and Uhler]{dif}
Yuhao Wang, Chandler Squires, Anastasiya Belyaeva, and Caroline Uhler.
\newblock Direct estimation of differences in causal graphs.
\newblock In S.~Bengio, H.~Wallach, H.~Larochelle, K.~Grauman, N.~Cesa-Bianchi,
  and R.~Garnett, editors, \emph{Advances in Neural Information Processing
  Systems}, volume~31. Curran Associates, Inc., 2018.

\end{thebibliography}

\appendix
\section{Appendix}
\subsection{Proof of Theorem \ref{iden_thm}}\label{A1}
\begin{proof}
The ``if'' direction is given by Corollary \ref{wij_eq_col}. For the ``only if'' direction, we distinguish several cases for the set $A_i$. The arguments for the corresponding different cases of $A_j$ are analogous.  In each case, we construct a set of parameters such that the considered rational equation in \eqref{wij_eq} does not hold.

\begin{enumerate}[label=\alph*)]
\item $\exists\ k\in\pa(i)\backslash A_i$: We choose $\lambda_{ki}\not=0$ and set all other edge weights equal to zero.  Then since $k\notin A_i$, the trek rule implies that $\Sigma_{i, A_i}=0$.  Hence \eqref{wij_eq} yields that
\[
\sigma_{ii} = \sigma_{jj}-\Sigma_{j,A_j}(\Sigma_{A_j,A_j})^{-1}\Sigma_{A_j,j} \le \sigma_{jj}.
\]
By the trek rule, it further holds that $\sigma_{jj}=\omega_{jj}$ and $\sigma_{ii}=\omega_{ii}+\lambda_{ki}^2\omega_{kk}$.  We arrive at the following contradiction:
\[
\sigma_{ii} \le \sigma_{jj} = \omega_{jj} = \omega_{ii} < \omega_{ii}+\lambda_{ki}^2\omega_{kk} = \sigma_{ii}.
\]
We conclude that $\pa(i)\subseteq A_i$.
\item $\exists\ k\in \de(i)\cap A_i$: 
There is then a directed path from $i$ to $k$ and as in the proof of Theorem \ref{wii-rev_thm}, we assume that $k$ was chosen such that this path is ``minimial''. In other words, the directed path is of the form $i\rightarrow m_1\rightarrow \cdots \rightarrow m_t\rightarrow k$ with $m_1,\dots,m_t\notin A_i$.  We proceed by distinguishing three subcases (illustrated in Figure~\ref{RT}):

\begin{figure}[htbp]
\centering
\label{fig:box}
\begin{minipage}[c]{0.3\textwidth}
\centering
\begin{tikzpicture}
[main_node/.style={draw=none,fill=none}]
\node[main_node] (1) at (-2.4, 3.4) {$i$};
\node[main_node] (2) at (-2.4, 2.4) {$m_1$};
\node[main_node] (5) at (-0.8, 0.8) {$m_{t-1}$};
\node (4) at ($(2)!.5!(5)$) {$\ldots$};
\node[main_node] (6) at (-0, 0) {$m_t$};
\node[main_node] (7) at (0, -1) {$k$};
\node[main_node] (8) at (-2, 0) {\textcolor{red}{$j$}};
\node[main_node] (9) at (-2.8, -0.8) {$\ldots$};

\draw [->, line width=1pt] (1) to (2);
\draw [->, line width=1pt] (2) to (4);
\draw [->, line width=1pt] (4) to (5);
\draw [->, line width=1pt] (5) to (6);
\draw [->, line width=1pt] (6) to (7);
\draw [->, line width=1pt, color=red] (8) to (9);
\end{tikzpicture}
\subcaption{}
\end{minipage}
\centering 
\label{fig:triangle}
\begin{minipage}[c]{0.3\textwidth}
\centering
\begin{tikzpicture}
[main_node/.style={draw=none,fill=none}]
\node[main_node] (1) at (-2.4, 3.4) {$i$};
\node[main_node] (2) at (-2.4, 2.4) {$m_1$};
\node[main_node] (5) at (-0.8, 0.8) {\textcolor{red}{$j$}};
\node (4) at ($(2)!.5!(5)$) {$\ldots$};
\node[main_node] (6) at (0, 0) {$\ldots$};
\node[main_node] (7) at (0.8, -0.8) {$m_t$};
\node[main_node] (8) at (2, -0.8) {$k$};
\node[main_node] (9) at (-0.8, -0.8) {$\ldots$};
\node[main_node] (10) at (-2, -0.8) {$k'$};

\draw [->, line width=1pt] (1) to (2);
\draw [->, line width=1pt] (2) to (4);
\draw [->, line width=1pt] (4) to (5);
\draw [->, line width=1pt, color=blue] (5) to (6);
\draw [->, line width=1pt] (6) to (7);
\draw [->, line width=1pt] (7) to (8);
\draw [->, line width=1pt,color=red] (6) to (9);
\draw [->, line width=1pt,color=red] (9) to (10);

\end{tikzpicture}
\subcaption{}
\end{minipage}
\centering 
\label{fig:triangle2}
\begin{minipage}[c]{0.3\textwidth}
\centering
\begin{tikzpicture}
[main_node/.style={draw=none,fill=none}]
\node[main_node] (1) at (-2.4, 3.4) {$i$};
\node[main_node] (2) at (-2.4, 2.4) {$m_1$};
\node[main_node] (5) at (-0.8, 0.8) {\textcolor{red}{$?$}};
\node (4) at ($(2)!.5!(5)$) {$\ldots$};
\node[main_node] (6) at (-0, 0) {\textcolor{red}{$j$}};
\node[main_node] (7) at (0.8, -0.8) {$\ldots$};
\node[main_node] (8) at (2, -0.8) {$k$};

\draw [->, line width=1pt] (1) to (2);
\draw [->, line width=1pt] (2) to (4);
\draw [->, line width=1pt] (4) to (5);
\draw [->, line width=1pt, color=red] (5) to (6);
\draw [->, line width=1pt] (6) to (7);
\draw [->, line width=1pt] (7) to (8);

\end{tikzpicture}
\subcaption{}
\end{minipage}
\caption{The three subcases when there exists a node $k\in\de(i)\cap A_i$.}
\label{RT}
\end{figure}
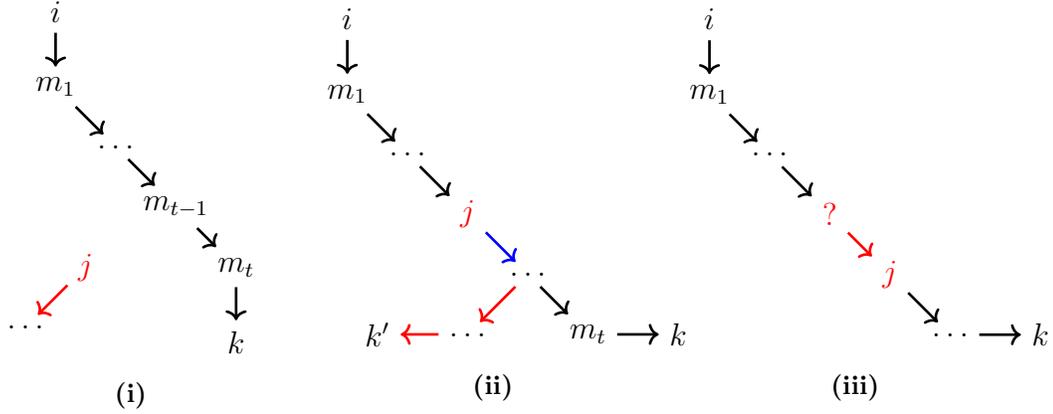

\begin{enumerate}[label=(\roman*)]
\item 
Suppose $k$ can be chosen such that there exists a directed path from $i$ to $k$ that is minimal in the above sense and does not intersect $j$.  Then we can set all edge weights zero except those on the path. As in the proof of Theorem \ref{wii-rev_thm}, we have $\sigma_{ii}=\omega_{ii}=\omega_{jj}=\sigma_{jj}$ and  find a contradiction because under equation~\eqref{wij_eq}, 
\begin{align*}
\omega_{ii}=\sigma_{ii}-\Big(\lambda_{im_1}\lambda_{m_tk}\prod_{s=2}^t\lambda_{m_{s-1}m_s}\Big)^2\frac{1}{\sigma_{kk}}<\sigma_{jj}=\omega_{jj}.
\end{align*}
\item Next, consider the case where every minimal directed path from $i$ to a node $k\in\de(i)\cap A_i$ contains the node $j$ and where in addition $A_j\cap \de(j)\not=\emptyset$.  Let $k'\in\de(j)\cap A_j$.  Then there exists a directed path from $j$ to $k'$.  It follows that in this subcase $j$ must be in $\de(i)$.  Since the graph is a DAG, the considered directed path from $j$ to $k'$ may not contain $i$.  Hence, we encounter exactly the situation of subcase (i), but with the role of $i$ and $j$ switched.  Hence, also in this case we can construct a counterexample to equation~\eqref{wij_eq}.
\item The remaining subcase is that every minimal directed path from $i$ to a node $k\in\de(i)\cap A_i$ contains the node $j$, and that these paths intersect $A_j$ only after they have visited $j$.  Select one such minimal directed path.  If the node preceding $j$ on the path is not in $A_j$, we can reduce the problem to case (a) by switching $i$ and $j$ ($\pa(j)\backslash A_j\neq\emptyset$). Otherwise, we set all edge weights zero except those on the considered minimal path. Let $A_j'$ be the intersection of $A_j$ and the nodes on the path. In the new DAG with only edges in the directed path, the set $A_j'$ satisfies that $\pa(j)\subseteq A_j'\subseteq V\backslash \de(j)$, and thus
\begin{align*}
\omega_{jj}=\sigma_{jj}-\Sigma_{j,A_j'}(\Sigma_{A_j',A_j'})^{-1}\Sigma_{A_j',j}=\sigma_{jj}-\Sigma_{j,A_j}(\Sigma_{A_j,A_j})^{-1}\Sigma_{A_j,j}.   
\end{align*}
However, computing the left hand side of \eqref{wij_eq} leads to a strict inequality.
\begin{align*}
\sigma_{ii}-\Sigma_{i,A_i}(\Sigma_{A_i,A_i})^{-1}\Sigma_{A_i,i}&=\omega_{ii}-\Big(\lambda_{im_1}\lambda_{m_tk}\prod_{s=2}^t\lambda_{m_{s-1}m_s}\Big)^2\frac{1}{\sigma_{kk}}<\omega_{ii}\\
&=\omega_{jj}=\sigma_{jj}-\Sigma_{j,A_j}(\Sigma_{A_j,A_j})^{-1}\Sigma_{A_j,j}. 
\qedhere
\end{align*}
\end{enumerate}
\end{enumerate}
\end{proof}

\subsection{``Only if'' part of Proposition \ref{ci_sndcomp}}\label{A2}
\begin{proof}
If $i$ and $j$ are $d$-connected given $S$, then there exists a path $q$ between $i$ and $j$, on which every collider is in $S$ (recall that our convention allows a path to visit the same node more than once).  We denote the set of all these colliders by $S'=\{z_1,z_2,\dots,z_k\}\subseteq S$; see Figure~\ref{fig:active:path} for an illustration.  
In order to form a covariance matrix in $M_{G,\Pi}$, we assign the same weight $\rho\in(0,1)$ to all edges of the path $q$ and set  all other edge weights zero.  We set all error variances $\omega_{ii}=1$.  Let $\Lambda$ and $\boldsymbol{w}$ be the resulting choice of parameters, and let $\Sigma=\phi_G(\Lambda,\boldsymbol{\omega})$ the associated covariance matrix.   

\begin{figure}[htbp]
\centering
\begin{tikzpicture}
[main_node/.style={draw=none,fill=none}]
\node[main_node] (1) at (0, 0) {$i$};
\node[main_node] (2) at (1.2, 0) {$\ldots$};
\node[main_node] (3) at (2.4, 0) {$h_1$};
\node[main_node] (4) at (2.4, -1.2) {$\vdots$};
\node[main_node] (5) at (2.4, -2.4) {$z_1$};

\node[main_node] (6) at (3.6, 0) {$\ldots$};
\node[main_node] (7) at (4.8, 0) {$h_2$};
\node[main_node] (8) at (4.8, -1.2) {$\vdots$};
\node[main_node] (9) at (4.8, -2.4) {$z_2$};

\node[main_node] (10) at (6, 0) {$\ldots$};
\node[main_node] (11) at (7.2, 0) {$\ldots$};
\node[main_node] (12) at (8.4, 0) {$h_k$};
\node[main_node] (13) at (8.4, -1.2) {$\vdots$};
\node[main_node] (14) at (8.4, -2.4) {$z_k$};

\node[main_node] (15) at (9.6, 0) {$\ldots$};
\node[main_node] (16) at (10.8, 0) {$j$};

\draw [->, line width=1pt] (1) to (2);
\draw [->, line width=1pt] (2) to (3);
\draw [->, line width=1pt] (3) to (4);
\draw [->, line width=1pt] (4) to (5);

\draw [<-, line width=1pt] (3) to (6);
\draw [->, line width=1pt] (6) to (7);
\draw [->, line width=1pt] (7) to (8);
\draw [->, line width=1pt] (8) to (9);

\draw [<-, line width=1pt] (7) to (10);
\draw [->, line width=1pt] (11) to (12);
\draw [->, line width=1pt] (12) to (13);
\draw [->, line width=1pt] (13) to (14);

\draw [->, line width=1pt] (15) to (12);
\draw [->, line width=1pt] (16) to (15);
\end{tikzpicture}
\caption{\label{fig:active:path} The active path $q$.}
\end{figure}
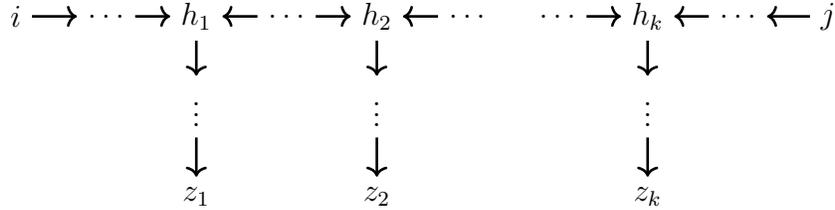

By the trek rule, the diagonal entries of $\Sigma=(\sigma_{kl})$ satisfy that 
\begin{align*}
   \sigma_{ii}&=\sigma_{jj}=1\quad\text{and}\quad \sigma_{kk}=1\; \forall\, k\notin S\setminus S',
\end{align*}
because the fact that $i$ and $j$ are $d$-connected given $S$ implies that the only nodes that are both in $S$ and on the path $q$ are the colliders in the set $S'$.  Next, notice that there exists a unique nonzero trek between each pair of consecutive nodes in the sequence $i\equiv z_0,z_1,z_2,\dots, z_k, z_{k+1}\equiv j$.  Let $r_t$ be the number of edges on the segment of $q$ that goes from $z_t$ to $z_{t+1}$.  By the trek rule, for all $t=0,\dots,k$,
\begin{align*}
    \sigma_{z_t,z_{t+1}} &= \rho^{r_t}.  
\end{align*}
 Ordering the nodes as $i,z_1,\dots, z_k,j$ followed by the nodes in $S\setminus S'$, we obtain that 
\begin{align}
\label{eq:Sigma'}
\Sigma_{ijS,ijS}=
\begin{pmatrix}
\begin{array}{cccccc|c}
1 & \rho^{r_0} & 0  & \cdots & 0 & 0 &\\
\rho^{r_0} & \sigma_{z_1,z_1} & \rho^{r_1}  & \cdots & 0 & 0 &\\
0 & \rho^{r_1} & \sigma_{z_2,z_2}  & \ddots  & 0 & 0 &\\
\vdots & \vdots & \ddots  & \ddots & \ddots & \vdots & O\\
0 & 0 & 0 & \ddots & \sigma_{z_k,z_k} & \rho^{r_k} & \\
0 & 0 & 0 & \cdots & \rho^{r_k} & 1 & \\ \hline
  &   & O & & & & I_{S\backslash S'}\\
\end{array}
\end{pmatrix}.
\end{align}
Now observe that $\det(\Sigma_{iS,jS}) = \rho^{\sum_{t=0}^k r_t}\not=0$.
\end{proof}

\subsection{Proof of Theorem \ref{alg_thm}}\label{A3}
\begin{proof}
Algorithm \ref{alg:1} builds upon the work of \citet{meek1995causal} who shows how to construct the CPDAG of an equivalence class when provided a set of conditional independence relations and arbitrary background knowledge about the edge orientations.  His general algorithm first constructs the classical CPDAG by reading off unshielded colliders and propagating rules R1, R2, R3.  Next, the general algorithm iteratively adds each edge from background knowledge and applies all rules R1, R2, R3, R4 to the 1-edge changes. Theorems 2-4 in \citet{meek1995causal} prove the correctness of the general algorithm. 

The application of R1-R3 before inserting background knowledge creates the classical CPDAG for known conditional independence relations and without extra information (it is the CPDAG under partition $\Pi_{\min}=\{\{i\}:i\in V\}$). In our setup, we start with a DAG $G$ in the equivalence class
and determine directly the skeleton and unshielded colliders and the classical CPDAG via rules R1-R3.

The partial homoscedasticity encoded in the given partition $\Pi$ now provides special `background knowledge' that fixes the orientation of all the edges with one endpoint at  special nodes.  As we show in the remainder of this proof, when we insert this special knowledge into the classical CPDAG, the situations of R3 and R4 in \citet{meek1995causal} cannot arise.  It thus suffices to apply only R1 and R2, and we can insert all the background knowledge simultaneously, because we know that all extra information is compatible and the desired CPDAG always exists.

For our proof of the correctness of the simplifications in Algorithm \ref{alg:1} over Meek's general procedure, recall that the equal variance constraints give the adjacency directions of all nodes whose block has size at least 2. The set $\textbf{R}$ consists of edges incident to these nodes, and the set $\textbf{F}$ consists of the reversal of the edges in $\textbf{R}$.  We then argue as follows.

\begin{enumerate}[label=(\roman*)]
\item First, we know there is at least one DAG in the equivalence class, so the general algorithm will not fail. That means the background knowledge check S1 and S1$'$ are redundant. We can just iteratively perform S2, S3 and S4 and obtain the same result.
\item Next, notice that we can add all edges in $\textbf{R}$ simultaneously and close the orientations sequentially. Indeed, every newly oriented edge is dependent on some of the background knowledge. As long as all dependencies are added, the edge will be oriented without conflicts. Either adding edges sequentially or simultaneously would finally cover all dependencies of each orientable edge, and results in the same final output.
\item Finally, we claim that only the rules R1 and R2 become applicable in the orientation propagation step S3 of our algorithm. 
Indeed, there is an unshielded collider triple in R3, but the propagation with background knowledge does not make any new collider triples, otherwise the output CPDAG cannot have same conditional independence statements as the DAG itself. Hence, any pattern of R3 must have been obtained in the initial phase of constructing the classical CPDAG, and will not appear in the last propagation phase. 

\begin{figure}[htbp]
\begin{minipage}[c]{0.48\textwidth}
\centering
\begin{tikzpicture}
[main_node/.style={draw=none,fill=none, inner sep=1.5pt}]
\node[main_node] (1) at (1, 0) {$i_1$};
\node[main_node] (2) at (2, 1) {$i_2$};
\node[main_node] (3) at (3, 0) {$i_3$};
\node[main_node] (4) at (2, -1) {$i_4$};
\node[main_node] (5) at (4.414, 0) {$l$};

\draw [-, line width=1pt] (1) to (2);
\draw [-{Stealth}, line width=1pt] (4) to (1);
\draw [-, line width=1pt] (2) to (3);
\draw [-{Stealth}, line width=1pt] (3) to (4);
\draw [-, line width=1pt] (2) to (4);
\draw [-{Stealth}, line width=1pt, color=red] (5) to (3);
\draw [-{Stealth}, line width=1pt, color=red, dashed] (5) to (2);
\draw [-{Stealth}, line width=1pt, color=red, dashed] (5) to (4);
\end{tikzpicture}
\caption{$i_3\rightarrow i_4$ from R1.}
\label{byR1}
\end{minipage}
\begin{minipage}[c]{0.48\textwidth}
\centering
\begin{tikzpicture}
[main_node/.style={draw=none,fill=none, inner sep=1.5pt}]
\node[main_node] (1) at (1, 0) {$i_1$};
\node[main_node] (2) at (2, 1) {$i_2$};
\node[main_node] (3) at (3, 0) {$i_3$};
\node[main_node] (4) at (2, -1) {$i_4$};
\node[main_node] (5) at (4.414, 0) {$l$};

\draw [-, line width=1pt] (1) to (2);
\draw [-{Stealth}, line width=1pt] (4) to (1);
\draw [-, line width=1pt] (2) to (3);
\draw [-{Stealth}, line width=1pt] (3) to (4);
\draw [-, line width=1pt] (2) to (4);
\draw [-{Stealth}, line width=1pt, color=red] (3) to (5);
\draw [-{Stealth}, line width=1pt, color=red, dashed] (5) to (2);
\draw [-{Stealth}, line width=1pt, color=red] (5) to (4);
\end{tikzpicture}
\caption{$i_3\rightarrow i_4$ from R2.}
\label{byR2}
\end{minipage}
\end{figure}

For R4, consider the first time that its pattern appears in the propagation phase. The orientation $i_3\rightarrow i_4$ is not obtained in the classical CPDAG phase as otherwise $i_4\rightarrow i_1$ would have also been oriented and the pattern of R4 appears in the classic CPDAG phase, which is a contradiction. If $i_3\rightarrow i_4$ results from background knowledge directly, then we know the orientations of all adjacencies of either $i_3$ or $i_4$, which will orient $i_2-i_3$ or $i_2-i_4$. This is a contradiction. Figure \ref{byR1} depicts the case of $i_3\rightarrow i_4$ obtained from R1: unshielded triple $l\rightarrow i_3 - i_4$. The edge $l\rightarrow i_2$ must exist to keep $i_2-i_3$ not oriented, consequently the undirected edge $i_2-i_4$ implies the adjacency between $l$ and $i_4$. The triple $(l, i_3, i_4)$ is shielded, contradicting the pattern of R1. Figure \ref{byR2} illustrates the case of $i_3\rightarrow i_4$ obtained from R2. To keep $i_2-i_4$ not oriented, the edge $l\rightarrow i_2$ must exist. But then $i_2-i_3$ can be oriented as $i_3\rightarrow i_2$, which is again a contradiction.
\end{enumerate}

In conclusion, we have proved that our modification to the general algorithm for equal variance constraints background knowledge is correct.
\end{proof}

\end{document}